\documentclass[10pt,a4paper]{article}
\usepackage[svgnames]{xcolor}
\usepackage{mdframed}
\usepackage[utf8]{inputenc}
\usepackage{amsmath, amsthm, amssymb, bbm, bm}
\usepackage{stmaryrd}
\usepackage{tikz-cd}
\usepackage{enumerate}
\usepackage{appendix}
\usepackage[backend=bibtex,style=ieee]{biblatex}
\usepackage[hidelinks]{hyperref}

\usepackage{color}

\definecolor{mygray}{gray}{0.75}

\newtheorem{tvrz}{Proposition}[section]
\newtheorem{lemma}[tvrz]{Lemma}
\newtheorem{theorem}[tvrz]{Theorem}

\theoremstyle{definition}
\newtheorem{definice}[tvrz]{Definition}
\theoremstyle{remark}
\newtheorem{rem}[tvrz]{Remark}
\theoremstyle{definition}
\newtheorem{mdexample}[tvrz]{Example}
\newenvironment{example}%
{\begin{mdframed}[topline=false, rightline=false, bottomline=false, linewidth=0.2em, linecolor=mygray, innerleftmargin=0.5em, innerrightmargin=0,leftmargin=-0.7em]\begin{mdexample}}%
{\end{mdexample}\end{mdframed}}

\def\^{\wedge}
\def\<{\langle}
\def\>{\rangle}
\def\~{\widetilde}

\def\D{\mathcal{D}}
\def\M{\mathcal{M}}
\def\cN{\mathcal{N}}
\def\C{\mathcal{C}}

\def\X{\mathfrak{X}}
\def\R{\mathbb{R}}
\def\N{\mathbb{N}}
\def\Z{\mathbb{Z}}

\def\da{\mathsf{a}}
\def\ds{\mathsf{s}}
\def\bbz{\mathbbm{z}}
\def\bbt{\mathbbm{t}}
\def\dr{\mathrm{d}}
\def\fp{\mathbf{p}}
\def\fm{\mathbf{m}}
\def\fq{\mathbf{q}}
\def\fr{\mathbf{r}}
\def\fI{\mathbf{I}}
\def\fA{\mathbf{A}}
\def\fM{\mathbf{M}}
\def\bbD{\mathbb{D}}
\def\sfD{\mathsf{D}}

\DeclareMathOperator{\Op}{\mathbf{Op}}

\def\1{\mathbbm{1}}
\def\ssm{\smallsetminus}
\def\fnula{\mathbf{0}}
\def\f1{\mathbf{1}}
\def\ddt{\left. \frac{\dr}{\dr{t}}\right|_{t=0} \hspace{-0.5cm}}

\def\ol{\overline}
\def\ul{\underline}

\bibliography{bib}
\begin{document}
\begin{flushright}
\today
\end{flushright}
\vspace{0.7cm}
\begin{center}

\baselineskip=13pt {\Large \bf{Flows on Graded Manifolds}\\}
 \vskip0.5cm
 {\large{Rudolf Šmolka$^{1}$, Jan Vysoký$^{1}$}}\\
 \vskip0.6cm
$^{1}$\textit{Faculty of Nuclear Sciences and Physical Engineering, Czech Technical University in Prague\\ Břehová 7, 115 19 Prague 1, Czech Republic, jan.vysoky@fjfi.cvut.cz}\\
\vskip0.3cm
\end{center}

\begin{abstract}
Flows of vector fields are an essential tool in differential geometry, with countless applications in both theory and practice. While they have been extensively studied for ordinary manifolds and supermanifolds, a treatment of flows in $\mathbb{Z}$-graded geometry is currently missing. $\mathbb{Z}$-graded manifolds constitute a natural generalization of supermanifolds, allowing functions in coordinates of arbitrary integer degree.

In this paper, flows of vector fields on $\mathbb{Z}$-graded manifolds are defined, and it is proved that every vector field admits a unique maximal flow (in the odd case, under the assumption that the vector field is homological). Vector fields invariant under flows are examined. Conditions under which the flows of two vector fields commute are investigated, and the interaction between flows corresponding to related vector fields is studied.

The main body of the paper is intentionally kept short, and we hope it is accessible to readers with only minimal prior knowledge of $\mathbb{Z}$-graded geometry.
\end{abstract}


\section*{Introduction}
In differential geometry, flows of vector fields constitute a fundamental bridge between local and global geometry. Their history traces back to works of S. Lie and H. Poincaré and since then, flows (integral curves, trajectories, orbits) appear throughout mathematics. In a nutshell, they allow one to integrate an infinitesimal object, a smoothly varying collection of tangential directions, to a one-parametric collection of global transformations of spaces. It is beyond the scope of this paper to list all consequences and uses of this observation. Flows relate infinitesimal and global symmetries in physics, provide a geometrical interpretation to time evolution in Hamiltonian mechanics, describe geodesic flows in Riemannian geometry, and appear in basically any theory with some dynamics there is. Last but not least, they are used as a tool in differential geometry. Frobenius theorem, Štefan--Sussmann integrability, orbits of Lie group actions, Darboux theorem in symplectic geometry, etc. 

Supermanifolds were introduced as a geometrical tool allowing for functions in anticommutative Grassmann variables. They appeared in various forms in the works of Berezin and Leites \cite{BerzeinLeites1975}, Kostant \cite{kostant1977graded}, Rogers \cite{rogers1980global} or DeWitt \cite{dewitt1992supermanifolds}. There are several well-established monographs on supergeometry \cite{carmeli2011mathematical, rogers2007supermanifolds, bartocci2012geometry, varadarajan2004supersymmetry, tuynman2004supermanifolds} and many others. Remarkably, out of all of them, it is  only \cite{tuynman2004supermanifolds} that treats flows of vector fields in a more systematic way. On the other hand, flows of vector fields on supermanifolds appeared, to some extent, in many research papers, see \cite{shander1980vector, bruzzo1985differential, hill1991super, dumitrescu2008superconnections}. The crucial references for our approach in this paper are \cite{monterde1988integral, monterde1993existence}. There, the authors construct local flows for (inhomogeneous) vector fields on supermanifolds under certain conditions, in a structured and rigorous way. Finally, the modern and self-contained paper on flows on supermanifolds is \cite{garnier2013integration}. 

$\Z$-graded manifolds form a natural generalization of supergeometry. Recently $\Z$-graded manifolds with local coordinates of arbitrary degrees \cite{Vysoky:2022gm, kotov2024category, fairon2017introduction} were introduced, thus extending the theory of non-negatively graded (or just $\N$-graded) manifolds \cite{Kontsevich:1997vb, severa2001some, Voronov:2019mav} and \cite{mehta2006supergroupoids,2011RvMaP..23..669C}. In the following, \textit{graded} will always mean \textit{$\Z$-graded}. To our knowledge, flows of vector fields on graded manifolds were not considered in any of the above references. 

The aim of this paper is to fill this void. We sincerely believe that the presented results are of vital importance for the future development of graded geometry, e.g. Frobenius theorem, Darboux theorem, Lie theory, graded extension of Hamiltonian mechanics, etc.

We have attempted to keep the main body of the text short and streamlined to have all important statements at hand. For this reason, we have moved all more complicated proofs into appendices. Only definitions and elementary results from ordinary differential geometry and ODE theory are used in proofs, making them accessible to readers with only a basic knowledge of graded geometry. 

The paper is organized as follows. 

We start by very brief introductory Section \ref{sec_notation} recalling basic definitions of graded geometry, conventions and notation used in this paper. 

In Section \ref{sec_flows}, we recall standard knowledge about flows of vector fields on ordinary smooth manifolds. We rephrase main statements in a way suitable for the purposes of graded geometry. 

In Section \ref{sec_gflows}, we define flows of vector fields on graded manifolds. We show that every flow has its unique infinitesimal generator. We argue that for degree zero case, every vector field has an underlying ordinary vector field. It turns out that if a vector field is an infinitesimal generator of some flow, its underlying vector field is the infinitesimal generator of the underlying ordinary flow. We then formulate the main statement of this paper, called the \textit{fundamental theorem on flows on graded manifolds}, which claims that every (homological in the odd degree case) vector field has a unique maximal flow. We conclude the section by explicitly calculating the maximal flow of the Euler vector field. 

Section \ref{sec_invariance} first defines what it means for a vector field to be invariant under some flow. We show that in a degree zero case, this can be restated as an actual invariance with respect to a family of local graded diffeomorphisms. We continue by proving that every vector field is invariant under its own flow. We conclude by showing that a given vector field is invariant under the flow of the other, if and only if their graded commutator vanishes.  

In Section \ref{sec_commuting}, we first discuss ordinary geometry and try to understand where two flows of ordinary vector fields can commute. This leads us to the introduction of commuting domains - special open sets where flows of two ordinary \textit{commuting} vector fields commute, and we investigate their basic properties. We recall the standard theorem about commuting flows in ordinary geometry. We define the notions of maximal commuting domains and commuting flows in graded setting. Finally, we prove the theorem saying that flows associated to two vector fields commute, if and only if the vector fields commute.

We conclude by Section \ref{sec_related}. Suppose one has two vector fields on two graded manifolds with a graded smooth map between them. If the two vector fields are related by the said graded smooth map, one can ask what is the corresponding property of their flows. It turns out that the answer is very natural: the map becomes equivariant with respect to the (locally defined) actions of the Abelian graded Lie group induced by the two flows. In fact, the converse statement is also true.

Let us briefly mention the structure of appendices. In Appendix \ref{app_proof0}, we prove the fundamental theorem on flows on graded manifolds for degree zero vector fields. This is by far the most involved part of this paper. The proof has several stages, hence it splits into six self-explanatory subsections. In Appendix \ref{app_proofneq0}, the proof of the theorem is given for non-zero degree vector fields. It turns out that it is essentially just a combinatorics for coefficients corresponding to powers of ``time evolution'' variable, hence it is much shorter. In Appendix \ref{app_invariance}, we prove the theorem claiming that a vector field is invariant with respect to the flow of another vector field, if and only if they commute. The complicated bit is again the one for degree zero flows. Finally, in Appendix \ref{app_commuting}, we the basic facts about commuting domains are proved, followed by the proof of the theorem claiming that flows of two vector fields commute, if and only if the vector fields commute. Not surprisingly, this is significantly more involved if both vector fields have degree zero.

\section*{Acknowledgments}
The research of was supported by grant GAČR 24-10031K. The authors are also grateful for a financial support from MŠMT under grant no. RVO 14000 from the Grant Agency of the Czech Technical University in Prague under grant no. SGS25/163/OHK4/3T/14. 

\section{Graded manifolds: conventions and notation} \label{sec_notation}
Before proceeding to the main body of the text, let us give the reader an overview of the conventions and notation used within. A good choice of notation should facilitate readability and comprehension, while allowing precise statements to be worded and proved. As these two requirements are, to some extend, mutually exclusive, every notation is a compromise. The field of graded geometry is, furthermore, notorious for varying conventions and definitions. To maintain consistency and avoid ambiguity, this text is build upon the definitions and theorems from \cite{Vysoky:2022gm}.

Thus a graded manifold is defined as a pair $\M = (M, \C^\infty_\M)$, where $M$ is a (Hausdorff, second countable) topological space, and $\C^\infty_\M$ is a sheaf of graded commutative associative unital algebras; it is also called the structure sheaf or the sheaf of graded functions. Consequently, for any open $U \subseteq M$ there is a set of graded functions $\C^\infty_\M(U)$ with the appropriate graded algebra structure. As open subsets are used in abundance throughout the text, we also use the notation $\Op(M)$ for the set of all open subsets of $M$, and $\Op_p(M)$ for the set of all open neighborhoods of some $p\in M$. Every graded function $f \in \C^\infty_\M(U)$ is assigned a degree $|f| \in \Z$.

Of course, not every pair $\M = (M, \C^\infty_\M)$ is a graded manifold. For it to be so, it needs to be locally isomorphic to a suitable model, called the graded domain $V^{(n_j)} = (V, \C^\infty_{(n_j)})$. Here $(n_j)$ is a sequence of non-negative integers, only finitely many of which are non-zero, $V \in \Op(\R^{n_0})$, and the model sheaf $\C^\infty_{(n_j)}$ assigns to every $W \in \Op(V)$ a set of formal power series in graded variables and with smooth functions on $W$ as coefficients. See Appendix \ref{subsection_setting_the_notation} for more. By only considering some $U \in \Op(M)$ and its open subsets, one obtains the open submanifold $\M|_U := (U, \C^\infty_\M|_{U})$.

Through the local isomorphisms of $\M$ with graded domains, the topological space $M$ gains the structure of a smooth manifold, which we call the underlying smooth manifold of $\M$. Morphisms of graded manifolds, or graded smooth maps, are pairs $\phi = (\ul{\phi}, \phi^\ast) \colon \M \to \cN$, where $\ul{\phi} \colon M \to N$ is a smooth map and $\phi^\ast \colon \C^\infty_\cN \to \ul{\phi}_\ast(\C^\infty_\M)$ is a morphism of sheaves of graded algebras. The sheaf morphism $\phi^\ast$ itself is a collection of graded algebra maps $\phi^\ast_U : \C^\infty_\cN(U) \to \C^\infty_\M(\ul{\phi}^{-1}(U))$ for every $U \in \Op(N)$, which behaves naturally under restrictions. Note that for a graded function $f \in \C^\infty_\cN(U)$ we often write $\phi^\ast(f)$ instead of the more correct $\phi_U^\ast(f)$. Using restrictions on the sheaf $\C^\infty_\M$ we can restrict a graded smooth map $\phi : \M \to \cN$ to $\phi|_{U} : \M|_U \to \cN|_V$ for any $U \in \Op(M)$ and $V \in \Op(N)$ such  that $\ul{\phi}(U) \subseteq V$, by setting $(\phi|_U)^\ast(f) := \phi^\ast(f)|_{U}$.

 One can view a smooth manifold $M$ as a (trivially) graded manifold. If $M$ is the underlying smooth manifold of $\M$, then there is a canonical graded smooth map $i_M : M \to \M$ with identity as the underlying map, and with the pullback given by $i^\ast_M(f) = \ul{f}$, which is called the body of the function $f$. In local coordinates, $\ul{f}$ is given simply by the coefficient of $f$ with no graded variable. Via the graded partition of unity, $i^\ast_M$ is shown to be surjective. In contrast, there is no canonical morphism $\M \to M$.

Vector fields on a graded manifold are most readily defined as derivations of the algebra of graded functions. Thus, a linear map $X : \C^\infty_\M(U) \to \C^\infty_\M(U)$ of degree $|X| \in \Z$ is a vector field on some $U \in \Op(M)$, denoted as $X \in \X_\M(U)$, if it satisfies the graded Leibniz rule $X(fg) = X(f)g + (-1)^{|X||f|}f X(g)$ for every $f,g \in \C^\infty_\M(U)$. The assignment $\X_\M : U \mapsto \X_\M(U)$ is a sheaf on $M$ of $\C^\infty_\M$-modules, which can again be shown using the graded partition of unity. In the case of the graded domain $\M = V^{(n_j)}$, we denote $\X_{(n_j)} := \X_{V^{(n_j)}}$.

\section{Flow domains and vector fields} \label{sec_flows}
We will first recall some basic notions of ordinary differential geometry. Let $M$ be a smooth manifold. We say that an open subset $D \subseteq M \times \R$ is a \textbf{flow domain on $M$}, if for every $m \in M$, the set
\begin{equation} \label{eq_Dmsubsetdomain}
D^{(m)} := \{ t \in \R \mid (m,t) \in D \} 
\end{equation}
is an open interval containing $0$. For any flow domain $D$ and $t \in \R$, the set 
\begin{equation} \label{eq_D(t)subset}
D_{(t)} := \{ m \in M \mid (m,t) \in D \}
\end{equation}
is an open (and possibly empty) subset of $M$. A smooth map $\theta_{0}: D \rightarrow M$ is called a \textbf{flow on $M$}, if $D$ is a flow domain and $\theta_{0}$ has the following two group-like properties:
\begin{enumerate}[(i)]
\item For any $m \in M$, one has $\theta_{0}(m,0) = m$;
\item For any $m \in M$, one has 
\begin{equation}
\theta_{0}( \theta_{0}(m,s),t) = \theta_{0}(m, s+t),
\end{equation}
for all $s,t \in \R$ such that both sides of the equation are well-defined. 
\end{enumerate}
Let us now make the following simple observation:
\begin{lemma} \label{lem_ordinaryflow}
Let $\theta_{0}: D \rightarrow M$ be a flow on $M$. Let us consider the following sets:
\begin{align}
D' := & \ \{ ((m,s),t) \in D \times \R \mid (\theta_{0}(m,s),t) \in D \}, \\ \label{eq_Dprimesubset}
D'' := & \ \{ (m,(s,t)) \in M \times (\R \times \R) \mid (m, s+t) \in D \}, \\
D^{\bullet} := & \  D' \cap \da_{0}^{-1}(D''), 
\end{align}
where $\da_{0}: (M \times \R) \times \R \rightarrow M \times (\R \times \R)$ is the canonical diffeomorphism. Then all three sets are open and the properties (i) and (ii) of the flow are equivalent to the commutativity of the diagrams
\begin{equation} \label{eq_theta0diagrams}
\begin{tikzcd} 
M \arrow{r}{(\1_{M},0_{M})} \arrow{rd}{\1_{M}} &[2em] D \arrow{d}{\theta_{0}} \\
& M 
\end{tikzcd}, \; \; 
\begin{tikzcd}
D^{\bullet} \arrow{d}{\da_{0}} \arrow{r}{\theta_{0} \times \1_{\R}} &[2em] D \arrow{dd}{\theta_{0}}\\
D'' \arrow{d}{\1_{M} \times +} & \\
D \arrow{r}{\theta_{0}} & M 
\end{tikzcd},
\end{equation}
where $0_{M}: M \rightarrow \R$ is the smooth map $0_{M}(m) := 0$ and $+: \R \times \R \rightarrow \R$ is the addition map $+(s,t) = s + t$. We do not explicitly write restrictions of maps. 
\end{lemma}
In other words, the flow $\theta_{0}$ is just a smooth action of the Lie group $(\R,+)$, except that the flow domain $D$ is not necessarily the whole product $M \times \R$. 

Let $\partial_{t} \in \X_{\R}(\R)$ denote the coordinate vector field corresponding to the global coordinate $t$ on $\R$. We will write $1 \otimes \partial_{t}$ for its lift to the product $M \times \R$. We will use the same symbol also for its restriction to any flow domain $D$. For any flow $\theta_{0}$ on $M$, there is a unique vector field $X_{0}$ on $M$ which is $\theta_{0}$-related to $1 \otimes \partial_{t}$, that is 
\begin{equation} \label{eq_X0theta0related}
(1 \otimes \partial_{t}) \circ \theta^{\ast}_{0} = \theta^{\ast}_{0} \circ X_{0}.
\end{equation}
This vector field is called the \textbf{infinitesimal generator of the flow $\theta_{0}$}. It follows from (\ref{eq_X0theta0related}) and the first diagram (\ref{eq_theta0diagrams}) that $X_{0}$ can be written as 
\begin{equation}
X_{0} = (\1_{M},0_{M})^{\ast} \circ (1 \otimes \partial_{t}) \circ \theta_{0}^{\ast}. 
\end{equation}
In more conventional language, this means that $X_{0}|_{m} = \ddt \hspace{5mm} \theta_{0}(m,t)$. It is a well-known fact that to any vector field on $M$, there is a unique associated flow. For the sake of its generalization in this paper, let us recall the statement precisely.

\begin{theorem}[\textbf{Fundamental theorem on flows}] \label{thm_fundamentalordinary}
Let $X_{0}$ be a smooth vector field on a manifold $M$. Then there exists a unique flow $\theta_{0}: D \rightarrow M$ on $M$ having the following properties:
\begin{enumerate}[(i)]
\item $X_{0}$ is the infinitesimal generator of $\theta_{0}$.
\item $\theta_{0}$ is maximal, i.e. it cannot be extended to a larger flow domain containing $D$. 
\end{enumerate}
$\theta_{0}$ is called the \textbf{flow generated by $X_{0}$}. Moreover, for each $m \in M$, it satisfies the following:
\begin{enumerate}[({p}1)]
    \item The map $\theta^{(m)}_{0} := \theta_{0}(m,\cdot): D^{(m)} \rightarrow M$ is the maximal integral curve of $X_{0}$ starting at $m$.
    \item For every $s \in D^{(m)}$, one has $D^{(\theta_{0}(m,s))} = D^{(m)} - s := \{ t - s \mid t \in D^{(m)} \}$.
\end{enumerate}
\end{theorem}
For the proof, see e.g. Theorem 9.12 in \cite{lee2012introduction}. Observe that the property \textit{(p2)} implies that for the maximal flow domain $D$, one has $D^{\bullet} = D'$. See Lemma \ref{lem_ordinaryflow} for the notation.
\section{Flows on graded manifolds} \label{sec_gflows}
Let $\M$ be a graded manifold over an ordinary manifold $M$. We would like to generalize the notion of a flow. A viewpoint of Lemma \ref{lem_ordinaryflow} should give us a hint. First, by $\R[-k]$ we mean a degree shifted vector space $\R$ viewed as a graded manifold with a single variable $\bbt$ of degree $|\bbt| = -k$. It has a natural structure of a graded Lie group with operations
\begin{equation}
    +: \R[-k] \times \R[-k] \rightarrow \R[-k], \; \; 0: \{ \ast \} \rightarrow \R[-k], \; \; -: \R[-k] \rightarrow \R[-k].
\end{equation}
Explicitly, we have $+^{\ast}(\bbt) = \bbt_{1} + \bbt_{2}$, $0^{\ast}(\bbt) = 0$, $-^{\ast}(\bbt) = -\bbt$, where $(\bbt_{1},\bbt_{2})$ are global coordinate functions on the product graded manifold induced by $\bbt$. Note that the underlying manifold of $\R[-k]$ is $\R$ for $k = 0$ and a single point $\{ \ast \}$ for $k \neq 0$. 

\begin{definice}
    By a \textbf{flow domain $\D$ on $\M$ of degree $k$}, we mean a graded manifold
    \begin{equation}
        \D = \left\{ \begin{array}{cc} 
        (\M \times \R)|_{D}, \text{ where $D$ is a flow domain on $M$} & \text{ for } k = 0, \\
        \M \times \R[-k]  & \text{ for } k \neq 0. 
        \end{array}
        \right.
    \end{equation}
    In both cases, we will write $D$ for the underlying manifold of $\D$.
\end{definice}
\begin{definice}\label{def_graded_flow}
    Let $\D$ be a flow domain on $\M$ of degree $k$. A \textbf{flow on $\M$ of degree $k$} is a graded smooth map $\theta: \D \rightarrow \M$, such that the following diagrams commute:
    \begin{equation} \label{eq_flowdiagrams}
        \begin{tikzcd}
            \M \arrow{dr}{\1_{\M}} \arrow{r}{(\1_{\M},0_{\M})} &[2em] \D \arrow{d}{\theta} \\
            & \M
        \end{tikzcd}, \; \; 
        \begin{tikzcd}
            \D^{\bullet} \arrow{d}{\da} \arrow{r}{\theta \times \1_{\R[-k]}} &[2em] \D \arrow{dd}{\theta}\\
            \D'' \arrow{d}{\1_{\M} \times +}& \\
            \D \arrow{r}{\theta}& \M.
        \end{tikzcd}.
    \end{equation}
    Let us explain the notation:
    \begin{enumerate}[(i)]
        \item For $k = 0$, $\D^{\bullet}$ and $\D''$ are restrictions of the products $(\M \times \R) \times \R$ and $M \times (\R \times \R)$ to the open sets $D^{\bullet}$ and $D''$ defined as in Lemma \ref{lem_ordinaryflow} for $\theta_{0} := \ul{\theta}: D \rightarrow M$. 
        \item For $k \neq 0$, we let $\D^{\bullet} = (\M \times \R[-k]) \times \R[-k]$ and $\D'' = \M \times (\R[-k] \times \R[-k])$. 
        \item $0_{\M}: \M \rightarrow \R[-k]$ is the composition of the terminal arrow $\M \rightarrow \{ \ast \}$ with $0: \{ \ast \} \rightarrow \R[-k]$. We will often use the notation $\ds_{0} := (\1_{\M},0_{\M})$ and call it the \textbf{zero section} of $\D$. 
        \item $\da$ is the restriction of the canonical associator diffeomorphism. 
    \end{enumerate}
    We do not explicitly write restrictions of involved graded smooth maps. 
\end{definice}
\begin{rem}
    It is convenient to think of flows as actions of a graded Lie group $(\R[-k],+)$, except for $k = 0$, they are defined only on the open submanifold $\D$. 
\end{rem}
\begin{rem}
    For $k = 0$, the underlying smooth map $\ul{\theta}: D \rightarrow M$ is a flow on $M$. This is obtained by applying the body functor to the diagrams (\ref{eq_flowdiagrams}). For $k \neq 0$, the underlying map $\ul{\theta}: M \times \{ \ast \} \rightarrow M$ is just the trivial diffeomorphism $\ul{\theta}(m,\ast) = m$.
\end{rem}
We define the degree of a flow to match the degree of its infinitesimal generator. Indeed, let $\partial_{\bbt}$ be the coordinate vector field on $\R[-k]$ corresponding to $\bbt$. Its degree is $|\partial_{\bbt}| = -|\bbt| = k$. Let $1 \otimes \partial_{\bbt}$ denote its lift to $\M \times \R[-k]$. This is a degree $k$ vector field uniquely determined by 
\begin{equation}
    1 \otimes \partial_{\bbt} \sim_{p_{1}} 0, \; \; 1 \otimes \partial_{\bbt} \sim_{p_{2}} \partial_{\bbt},
\end{equation}
where $p_{1}: \M \times \R[-k] \rightarrow \M$ and $p_{2}: \M \times \R[-k] \rightarrow \R[-k]$ are the projections. We will use the same symbol for its restriction to $\D$. 
\begin{tvrz}
Let $\theta: \D \rightarrow \M$ be a flow on $\M$ of degree $k$. Then there is a unique vector field $X$ on $\M$ of degree $k$ satisfying
\begin{equation} \label{eq_partialxitheterelated}
    1 \otimes \partial_{\bbt} \sim_{\theta} X.
\end{equation}
$X$ is called the \textbf{infinitesimal generator} of $\theta$. It always satisfies $[X,X] = 0$. 
\end{tvrz}
\begin{proof}
    The requirement (\ref{eq_partialxitheterelated}) explicitly means that 
    \begin{equation} \label{eq_partialxitheterelated2} 
        (1 \otimes \partial_{\bbt}) \circ \theta^{\ast} = \theta^{\ast} \circ X,
    \end{equation}
    where $\theta^{\ast}: \C^{\infty}_{\M}(M) \rightarrow \C^{\infty}_{\D}(D)$ is the pullback induced by $\theta$. Let $\ds_{0} := (\1_{\M}, 0_{\M}): \M \rightarrow \M \times \R[-k]$. It follows from the first diagram in (\ref{eq_flowdiagrams}) and (\ref{eq_partialxitheterelated2}) that the only way to define $X$ is
    \begin{equation}
        X := \ds_{0}^{\ast} \circ (1 \otimes \partial_{\bbt}) \circ \theta^{\ast}. 
    \end{equation}
    We invite the reader to check that this defines a vector field on $\M$ of degree $k$. To check that such $X$ indeed fits into (\ref{eq_partialxitheterelated2}), one has to utilize the other diagram in (\ref{eq_flowdiagrams}) together with the fact that $\partial_{\bbt}$ is a left-invariant vector field on $(\R[-k],+)$. The proof is analogous to the construction of infinitesimal generators for general graded Lie group actions, see Proposition 3.22 in \cite{Smolka2023}.

    Finally, it follows from (\ref{eq_partialxitheterelated}) that $0 \sim_{\theta} [X,X]$, that is $\theta^{\ast} \circ [X,X] = 0$. Acting on both sides of this equation by $\ds_{0}^{\ast}$ and using the first diagram in (\ref{eq_flowdiagrams}), one finds $[X,X] = 0$. 
\end{proof}
\begin{rem}\label{rem_homological}
    We see that whenever $X$ is an infinitesimal generator of a flow on $\M$, it satisfies $[X,X] = 0$. If its degree $k$ is odd, this is a non-trivial statement. We see that for flows of odd degrees, their infinitesimal generators are \textit{homological vector fields}.
\end{rem}
We have already observed that for $k = 0$, the underlying map $\ul{\theta}: D \rightarrow M$ is a flow on $M$. How is its infinitesimal generator related to the one of $\theta$?
\begin{lemma} \label{lem_underlying}
    Let $X$ be a vector field on $\M$ of degree zero. Then there is a unique vector field $X_{0}$ on $M$ satisfying the equation
    \begin{equation} \label{eq_underlying}
        \ul{X(f)} = X_{0}(\ul{f}),
    \end{equation}
    for all $f \in \C^{\infty}_{\M}(M)$. In other words, one has $X_{0} \sim_{i_{M}} X$, where $i_{M}: M \rightarrow \M$ is the canonical embedding defined by $i^{\ast}_{M}(f) = \ul{f}$. $X_{0}$ is called the \textbf{underlying vector field} of $X$. 
\end{lemma}
\begin{proof}
    The body map $f \mapsto \ul{f}$ is surjective, so one can define $X_{0}$ using the formula (\ref{eq_underlying}). One only has to check that $\ul{f} = 0$ implies $\ul{X(f)} = 0$. This can be verified locally, so we may assume that $\M$ is a graded domain with coordinates $\{ \bbz^{\lambda} \}_{\lambda=1}^{n}$. One can write $X(f) = X^{\lambda} \cdot \partial_{\lambda}f$ and thus
    \begin{equation} \label{eq_ulXfunderlying}
        \ul{X(f)} = \ul{X}^{\lambda} \cdot \ul{\partial_{\lambda} f}.
    \end{equation}    
    Since $|X^{\lambda}| = |\bbz^{\lambda}|$, we have $\ul{X}^{\lambda} = 0$ whenever $|\bbz^{\lambda}| \neq 0$. Moreover, recall that $f$ is a formal power series in \textit{non-zero} degree coordinates with coefficients being smooth functions of \textit{zero-degree} coordinates. If $|\bbz^{\lambda}| = 0$, the operator $\partial_{\lambda}$ acts by partial differentiation on the coefficients of $f$. In particular, if $f$ has a trivial body, then so does $\partial_{\lambda} f$. This shows that $\ul{\partial_{\lambda} f} = 0$ whenever $|\bbz^{\lambda}| = 0$. Consequently, all terms of the sum (\ref{eq_ulXfunderlying}) vanish. This shows that $X_{0}$ is well-defined. It is easy to check that it is a smooth vector field on $M$. In fact, in coordinates it takes the form 
    \begin{equation}
        X_{0} = \ul{X}^{\lambda} \partial_{\lambda}.
    \end{equation}
    Note that the contribution from the summands where $|\bbz^{\lambda}| \neq 0$ is trivial. 
\end{proof}
\begin{rem}
    For $k = 0$, we will often denote the coordinate $\bbt$ on $\R[-0] \cong \R$ as $t$, for $k \neq 0$, we will usually denote $\bbt$ as $\tau$. 
\end{rem}
\begin{tvrz} \label{tvrz_underlyingflow}
    Let $\theta$ be a flow on $\M$ of degree $0$ with the infinitesimal generator $X$. Then the infinitesimal generator of the flow $\ul{\theta}$ on $M$ is precisely the underlying vector field $X_{0}$.
\end{tvrz}
\begin{proof}
    For any $f \in \C^{\infty}_{\M}(M)$, one has 
    \begin{equation}
        X_{0}(\ul{f}) = \ul{X(f)} = \ul{ \ds_{0}^{\ast}((1_{\M} \otimes \partial_{t})\theta^{\ast}(f))} = \ul{\ds_{0}}^{\ast}((1_{M} \otimes \partial_{t}) \ul{\theta}^{\ast}(\ul{f})),
    \end{equation}
    where we have used $\M$ and $M$ subscripts to distinguish between the vector fields $1_{\M} \otimes \partial_{t}$ on $\M \times \R$ and $1_{M} \otimes \partial_{t}$ on $M \times \R$, respectively. But $\ul{s_{0}}: M \rightarrow M \times \R$ is the zero section $(1_{M},0_{M})$ and the right-hand side is thus the infinitesimal generator of $\ul{\theta}$ applied on $\ul{f}$. Since the body map $f \mapsto \ul{f}$ is surjective, this proves the claim.
\end{proof}
We can now formulate the main statement of this paper.

\begin{theorem}[\textbf{Fundamental theorem on flows on graded manifolds}]\label{thm_funamental_extraordinary}
    Let $X$ be a vector field on $\M$ of degree $k$ satisfying $[X,X] = 0$. 

    Then there exists a unique flow $\theta: \D \rightarrow \M$ of degree $k$, having the following properties:
    \begin{enumerate}[(i)]
        \item $X$ is the infinitesimal generator of $\theta$.
        \item Suppose $\theta': \D' \rightarrow \M$ is any other flow satisfying $(i)$. Then $D' \subseteq D$ and $\theta' = \theta|_{D}$. We say that the flow domain $\D$ is \textbf{maximal}. 
    \end{enumerate}
     $\theta$ is called the \textbf{flow generated by $X$}. Moreover, for $k = 0$, the underlying map $\ul{\theta}: D \rightarrow M$ is the unique maximal flow $\theta_{0}$ of the underlying vector field $X_{0}$.
\end{theorem}
We have decided to move the proof of this theorem to Appendix \ref{app_proof0} (degree zero case) and to Appendix \ref{app_proofneq0} (nonzero degree case). Note that for $k \neq 0$, any flow of degree $k$ satisfying $(i)$ has the property $(ii)$. This is because the body manifold of the flow domain of degree $k \neq 0$ is always just $M \times \{\ast\}$. Let us finish this section with one interesting example.
\begin{example}
    Let $E \in \X_{\M}(M)$ be the Euler vector field on $\M$. This is a degree zero vector field defined by the formula
    \begin{equation}
        E(f) = |f| f,
    \end{equation}
    for all $f \in \C^{\infty}_{\M}(M)$. Its underlying vector field $E_{0}$ is trivial. Let us try and find the flow $\theta$ generated by $E$. Its underlying map $\ul{\theta}$ must be the flow generated by $E_{0} = 0$. We thus know that $D = M \times \R$ and $\ul{\theta} = p_{1}: M \times \R \rightarrow M$ is the canonical projection. The defining equation (\ref{eq_partialxitheterelated2}) reads
    \begin{equation} \label{eq_PDEforEuler}
        (1 \otimes \partial_{t}) \theta^{\ast}(f) = |f|\,  \theta^{\ast}(f),
    \end{equation}
    for all $f \in \C^{\infty}_{\M}(M)$. The first diagram in (\ref{eq_flowdiagrams}) can be interpreted as the initial condition for this differential equation in the form $\ds_{0}^{\ast}( \theta^{\ast}(f)) = f$ for all $f \in \C^{\infty}_{\M}(M)$. It is easy to see that (\ref{eq_PDEforEuler}) with this initial condition can be solved by
    \begin{equation}
        \theta^{\ast}(f) = \exp(t |f|) \cdot p_{1}^{\ast}(f),
    \end{equation}
    where $p_{1}: \M \times \R \rightarrow \M$ is the projection. We invite to reader to check that $\theta$ has the required properties of the flow. Observe that even in the case of the vanishing underlying vector field, flow equations can give non-trivial (systems of) ordinary differential equations.
\end{example}

    \section{Invariance under flows} \label{sec_invariance}
    Let us now discuss the following. Let $X \in \X_{\M}(M)$ be a vector field on $\M$ of degree $k$, and let $\theta: \D \rightarrow \M$ be the flow generated by $X$. Let $Y \in \X_{\M}(M)$ be another vector field. We say that $Y$ is \textbf{invariant under the flow $\theta$}, if 
    \begin{equation} \label{eq_invariantVF}
        Y \otimes 1 \sim_{\theta} Y,
    \end{equation}
    where $Y \otimes 1$ is the canonical lift of $Y$ to $\D$. 
    
    Let us now discuss this condition for $k = 0$. It turns out that it can be interpreted in a more familiar fashion. For any given $t \in \R$, let $D_{(t)}$ be the (possibly empty) open subset (\ref{eq_D(t)subset}). Let $\ds_{t}: \M|_{D_{(t)}} \rightarrow (\M \times \R)|_{D}$ be the \textbf{section valued at $t$}, that is 
    \begin{equation}
        p_{1} \circ \ds_{t} = \iota_{(t)}, \; \; p_{2} \circ \ds_{t} = j_{(t)},
    \end{equation}
    where $\iota_{(t)}: \M|_{D_{(t)}} \rightarrow \M$ is the canonical inclusion, and $j_{(t)}: \M \rightarrow \R$ is the terminal arrow $\M \rightarrow \{t\}$ followed by the one-point embedding $\{t \} \rightarrow \R$. Since $\theta_{0}(D_{(t)}) \subseteq D_{(-t)}$, there is a unique graded smooth map $\theta_{(t)}: \M|_{D_{(t)}} \rightarrow \M|_{D_{(-t)}}$ fitting into the diagram
    \begin{equation} \label{eq_thetatmap}
        \begin{tikzcd}
            \M|_{D_{(t)}} \arrow{d}{\ds_{t}} \arrow[dashed]{r}{\theta_{(t)}} & \M|_{D_{(-t)}} \arrow{d}{\iota_{(-t)}} \\
            (\M \times \R)|_{D} \arrow{r}{\theta} & \M 
        \end{tikzcd}
    \end{equation}
    \begin{lemma} \label{lem_thetatisdiffeo}
        For each $t \in \R$ with $D_{(t)} \neq \emptyset$, $\theta_{(t)}$ is a graded diffeomorphism with the inverse $\theta_{(-t)}$. 
    \end{lemma}
    \begin{proof}
        Let us first consider a graded smooth map
        \begin{equation} \label{eq_Delta-}
            \Delta_{-} := (\1_{\M \times \R}, - \circ p_{2}): \M \times \R \rightarrow (\M \times \R) \times \R,
        \end{equation}
        where $-: \R \rightarrow \R$ is the smooth map $t \mapsto -t$ and $p_{2}: \M \times \R \rightarrow \R$ is the projection. Its underlying smooth map is just $(m,t) \mapsto ((m,t),-t)$. Observe that whenever $(m,t) \in D$, one has $((m,t),-t) \in D'$. This follows from \textit{(p2)} in Theorem \ref{thm_fundamentalordinary}. We can thus form a composition 
        \begin{equation}
        \Delta_{-} \circ \ds_{t}: \M|_{D_{(t)}} \rightarrow ((\M \times \R) \times \R)|_{D'}.
        \end{equation}
        Our intentions are to precompose both paths along the second commutative diagram in (\ref{eq_flowdiagrams}) with this map. It is not difficult to see that 
        \begin{equation}
            (\1_{\M} \times +) \circ \da \circ \Delta_{-} \circ \ds_{t} = \ds_{0}|_{D_{(t)}},
        \end{equation}
        so the composition with the left-lower path of the diagram gives 
        \begin{equation}
            \theta \circ (\1_{\M} \times +) \circ \da \circ \Delta_{-} \circ \ds_{t} = \ds_{0}|_{D_{(t)}} = \theta \circ \ds_{0}|_{D_{(t)}} =  \iota_{(t)} \circ \1_{\M|_{D_{(t)}}}.
        \end{equation}
        where $\iota_{(t)}: \M|_{D_{(t)}} \rightarrow \M$ is the inclusion. On the other hand, one has 
        \begin{equation}
            (\theta \times \1_{\R}) \circ \Delta_{-} \circ \ds_{t} = \ds_{-t} \circ \theta_{(t)}. 
        \end{equation}
        This can be seen easily by composing both side with projections onto $\M$ and $\R$, respectively. Hence the composition with the upper-right path of the second diagram in (\ref{eq_flowdiagrams}) with $\Delta_{-} \circ \ds_{t}$ gives 
        \begin{equation}
            \theta \circ (\theta \times \1_{\R}) \circ \Delta_{-} \circ \ds_{t} = \theta \circ \ds_{-t} \circ \theta_{(t)} = \iota_{(t)} \circ \theta_{(-t)} \circ \theta_{(t)}. 
        \end{equation}
        The commutativity of the second diagram in (\ref{eq_flowdiagrams}) and the fact that $\iota_{(t)}$ is an embedding of $\M|_{D_{(t)}}$ into $\M$ thus gives the formula
        \begin{equation}
            \theta_{(-t)} \circ \theta_{(t)} = \1_{\M|_{D_{(t)}}}.
        \end{equation}
        The reverse composition is obtained by replacing $t$ with $-t$. Hence $\theta_{(-t)}$ is the inverse to $\theta_{(t)}$. 
    \end{proof}
    For $k = 0$, we can now reformulate the invariance of $Y$ with respect to the flow $\theta$ of $X$. 
    \begin{tvrz} \label{tvrz_invariancefordegreezero}
       A vector field $Y \in \X_{\M}(M)$ is invariant under the flow $\theta$ of a degree zero vector field $X \in \X_{\M}(M)$, iff for all $t \in \R$ with $D_{(t)} \neq \emptyset$, one has
        \begin{equation} \label{eq_Yinvarethetatrelated}
            Y|_{D_{(t)}} \sim_{\theta_{(t)}} Y|_{D_{(-t)}}.
        \end{equation}
    \end{tvrz}
    \begin{proof}
        The equation (\ref{eq_invariantVF}) can be rewritten as 
        \begin{equation} \label{eq_Yinvariantequation}
            (Y \otimes 1) \circ \theta^{\ast} = \theta^{\ast} \circ Y.
        \end{equation}
        Compose both sides with the pullback $\ds^{\ast}_{t}: \C^{\infty}_{\M \times \R}(D) \rightarrow \C^{\infty}_{\M}(D_{(t)})$ and observe that
        \begin{equation}
            Y|_{D_{(t)}} \sim_{\ds_{t}} Y \otimes 1, \; \; \theta \circ \ds_{t} = \iota_{(-t)} \circ \theta_{(t)}, \; \; Y|_{D_{(-t)}} \sim_{\iota_{(-t)}} Y. 
        \end{equation}
        We obtain the equation
        \begin{equation}
         Y|_{D_{(t)}} \circ \theta^{\ast}_{(t)} \circ \iota^{\ast}_{(-t)} = \theta^{\ast}_{(t)} \circ Y|_{D_{(-t)}} \circ \iota^{\ast}_{(-t)}. 
        \end{equation}
        Now, although $\iota^{\ast}_{(-t)}$ is not necessarily surjective, we can already use this equation to deduce 
        \begin{equation} \label{eq_YthetatrelatedtoYequation}
            Y|_{D_{(t)}} \circ \theta^{\ast}_{(t)} = \theta^{\ast}_{(t)} \circ Y|_{D_{(-t)}},
        \end{equation}
        that is precisely the equation (\ref{eq_Yinvarethetatrelated}). See Lemma \ref{lem_relatedwithiV} in the appendix. This proves the only if part. Conversely, if (\ref{eq_Yinvarethetatrelated}) and thus (\ref{eq_YthetatrelatedtoYequation}) is true for all $t \in \R$ with $D_{(t)} \neq \emptyset$, we compose both its sides with $i^{\ast}_{(-t)}$ and reverse the above steps to obtain
        \begin{equation}
             \ds^{\ast}_{t} \circ (Y \otimes 1) \circ \theta^{\ast} = \ds^{\ast}_{t} \circ \theta^{\ast} \circ Y,
        \end{equation}
        for all $t \in \R$ with $D_{(t)} \neq \emptyset$. The claim now follows from Lemma \ref{lem_invariancecomposedwiths't}. 
    \end{proof}
    It turns out that vector fields are invariant under their own flows.
    \begin{tvrz} \label{tvrz_Xisinvariantunderitsflow}
        Let $X \in \X_{\M}(M)$ be any vector field. 

        Then $X$ is invariant under its own flow $\theta$. 
    \end{tvrz}
    \begin{proof}
        The commutativity of the second diagram in (\ref{eq_flowdiagrams}) can be written as 
        \begin{equation} \label{eq_flowdiagraminpullbacks}
            (\theta \times \1_{\R[-k]})^{\ast} \circ \theta^{\ast} = \da^{\ast} \circ (\1_{\M} \times +)^{\ast} \circ \theta^{\ast}. 
        \end{equation}
        Let us now act on both sides by the vector field $(1 \otimes \partial_{\bbt}) \otimes 1$. For the left-hand side, one finds
        \begin{equation}
            ((1 \otimes \partial_{\bbt}) \otimes 1) \circ (\theta \times \1_{\R[-k]})^{\ast} \circ \theta^{\ast} = (\theta \times \1_{\R[-k]})^{\ast} \circ (X \otimes 1) \circ \theta^{\ast},
        \end{equation}
        where we have used the fact that $1 \otimes \partial_{\bbt} \sim_{\theta} X$ implies $(1 \otimes \partial_{\bbt}) \otimes 1 \sim_{\theta \times \1_{\R[-k]}} X \otimes 1$. For the right-hand side, note that $(1 \otimes \partial_{\bbt}) \otimes 1$ is $\da$-related to $1 \otimes (\partial_{\bbt} \otimes 1)$, so 
        \begin{equation}
            ((1 \otimes \partial_{\bbt}) \otimes 1) \circ \da^{\ast} \circ (\1_{\M} \times +)^{\ast} \circ \theta^{\ast} = \da^{\ast} \circ (1 \otimes (\partial_{\bbt} \otimes 1)) \circ (\1_{\M} \times +)^{\ast} \circ \theta^{\ast} 
        \end{equation}
        Now, since $(\R,+)$ is Abelian, $\partial_{\bbt}$ is also right-invariant and we have $\partial_{\bbt} \otimes 1 \sim_{+} \partial_{\bbt}$. Hence
        \begin{equation}
            (1 \otimes (\partial_{\bbt} \otimes 1)) \circ (\1_{\M} \times +)^{\ast} = (\1_{\M} \times +)^{\ast} \circ (1 \otimes \partial_{\bbt}).
        \end{equation}
        We can use this together with $1 \otimes \partial_{\bbt} \sim_{\theta} X$ to get 
        \begin{equation}
        \begin{split}
        \da^{\ast} \circ (1 \otimes (\partial_{\bbt} \otimes 1)) \circ (\1_{\M} \times +)^{\ast} \circ \theta^{\ast} = & \ \da^{\ast} \circ (\1_{\M} \times +)^{\ast} \circ \theta^{\ast} \circ X \\
        = & \ (\theta \times \1_{\R[-k]})^\ast \circ \theta^{\ast} \circ X,
        \end{split}
        \end{equation}
        where we have utilized (\ref{eq_flowdiagraminpullbacks}). To summarize, by acting by $(1 \otimes \partial_{\bbt}) \otimes 1$ on (\ref{eq_flowdiagraminpullbacks}), we get
        \begin{equation}
            (\theta \times \1_{\R[-k]})^{\ast} \circ (X \otimes 1) \circ \theta^{\ast} = (\theta \times \1_{\R[-k]})^\ast \circ \theta^{\ast} \circ X.
        \end{equation}
        Finally, act on both sides by $(\ds_{0} \times \1_{\R[-k]})^{\ast}$ and use $(\theta \circ \1_{\R[-k]}) \circ (\ds_{0} \times \1_{\R-k]})|_{D} = \1_{\D}$. This gives
        \begin{equation}
            (X \otimes 1) \circ \theta^{\ast} = \theta^{\ast} \circ X.
        \end{equation}
        But this means precisely that $X$ is invariant under the flow $\theta$. 
    \end{proof}
    
    There exists a simple criterion for $Y$ to be invariant under the flow $\theta$ of $X$.
    \begin{theorem} \label{thm_invariance}
        Let $X,Y \in \X_{\M}(M)$ be vector fields on $\M$, and let $\theta$ be the flow of $X$. 

        Then $Y$ is invariant under the flow $\theta$, iff $[X,Y] = 0$. 
    \end{theorem}
    \begin{proof}
        Let us prove the easy direction first. Suppose that $Y$ is invariant under the flow $\theta$, that is $Y \otimes 1 \sim_{\theta} Y$. By definition of the flow, we have $1 \otimes \partial_{\bbt} \sim_{\theta} X$. But then 
        \begin{equation}
            [1 \otimes \partial_{\bbt}, Y \otimes 1] \sim_{\theta} [X,Y].
        \end{equation}
        The commutator on the left-hand side is easily seen to be zero, hence $0 \sim_{\theta} [X,Y]$, that is 
        \begin{equation}
            \theta^{\ast} \circ [X,Y] = 0.
        \end{equation}
        Now pull back both sides with $\ds_{0} = (\1_{\M},0_{\M})$ and use the first diagram in (\ref{eq_flowdiagrams}) to get $[X,Y] = 0$. The proof of the fact that $[X,Y] = 0$ makes $Y$ invariant under the flow $\theta$ is a bit more involved, especially in the $k = 0$ case. An interested reader can find it in Appendix \ref{app_invariance}.
    \end{proof}
\section{Commuting flows} \label{sec_commuting}
In this section, we would like to define what does it mean for two flows of vector fields on graded manifolds to commute. Even in ordinary geometry, this is not always done correctly. 

Let $X,Y \in \X_{M}(M)$ be vector fields on an ordinary manifold $M$, and let $\theta^{X}: D_{1} \rightarrow M$ and $\theta^{Y}: D_{2} \rightarrow M$ be the corresponding flows. One says that the two flows \textbf{commute}, if
\begin{equation} \label{eq_commuteordinary}
    \theta^{Y}(\theta^{X}(m,s),t) = \theta^{X}(\theta^{Y}(m,t),s).
\end{equation}
This equation is not well-defined for all $m \in M$ and $s,t \in \R$. Let us define open subsets
\begin{align}
    D'_{1} = & \ \{ ((m,s),t) \in D_{1} \times \R \mid \theta^{X}(m,s) \in D_{2} \}, \label{eq_D'1set} \\ 
    D'_{2} = & \ \{ ((m,t),s) \in D_{2} \times \R \mid \theta^{Y}(m,t) \in D_{1} \}. \label{eq_D'2set}
\end{align}
It follows that the equation is defined for $((m,s),t) \in D'_{1} \cap \sigma^{-1}(D'_{2})$, where $\sigma((m,s),t) := ((m,t),s)$ is the ``flip map''. It is a well-known fact that the flows $\theta^{X}$ and $\theta^{Y}$ commute, iff $[X,Y] = 0$. However, it turns out that (\ref{eq_commuteordinary}) is not necessarily true for all $((m,s),t) \in D'_{1} \cap \sigma^{-1}(D'_{2})$. To see why, see Exercise 9-19 in \cite{lee2012introduction}. This leads us to the following definition:
\begin{definice} \label{def_commutingdomain}
    Let $\theta^{X}: D_{1} \rightarrow M$ and $\theta^{Y}: D_{2} \rightarrow M$ be the flows of $X \in \X_{M}(M)$ and $Y \in X_{M}(M)$, respectively. Let $D'_{1}$ and $D'_{2}$ be the open sets (\ref{eq_D'1set}, \ref{eq_D'2set}), and let $\sigma((m,s),t) := ((m,t),s)$. We say that an open subset $\bbD \subseteq (\M \times \R) \times \R$ is a \textbf{commuting domain} of $X$ and $Y$, if
    \begin{enumerate}[(1)]
        \item $\bbD \subseteq D'_{1} \cap \sigma^{-1}(D'_{2})$.
        \item For each $m \in M$, the subset $\bbD^{(m)} = \{ (s,t) \mid ((m,s),t) \in \bbD \}$ has the following property: if $(s,t) \in \bbD^{(m)}$, then $\bbD^{(m)}$ contains the entire closed rectangle with corners $(0,0)$ and $(s,t)$.
    \end{enumerate}
\end{definice}
Let us state some basic facts about commuting domains:
\begin{lemma} \label{lem_commuting}
    \begin{enumerate}[(i)]
        \item For any $X,Y \in \X_{M}(M)$, there exists a commuting domain of $X$ and $Y$.
        \item Any finite intersection and any union of commuting domains of $X$ and $Y$ is a commuting domain of $X$ and $Y$.
        \item For each $X,Y \in \X_{M}(M)$, there exists their unique \textbf{maximal commuting domain}, that is a commuting domain containing all their commuting domains.
        \item If both $X$ and $Y$ are complete, their  maximal commuting domain is $(M \times \R) \times \R$. 
        \item Let $m \in M$. Suppose that there are open intervals $I,J \in \Op_{0}(\R)$, such that $((m,s),t) \in D'_{1} \cap \sigma^{-1}(D'_{2})$ for all $(s,t) \in I \times J$. Then $((m,s),t) \in \bbD$, where $\bbD$ is the maximal commuting domain of $X$ and $Y$. 
        \item Let $\bbD$ be any commuting domain. Let $((m,s),t) \in \bbD$. Then there is $U \in \Op_{m}(M)$ and open intervals $I,J \in \Op_{0}(\R)$, such that $((m,s),t) \in (U \times I) \times J \subseteq \bbD$. 
    \end{enumerate}
\end{lemma}
\begin{proof}
    See Appendix \ref{subsec_factscommuting}.
\end{proof}
The following statement is just a minor restatement of Theorem 9.44 in \cite{lee2012introduction}. Its version stated here can be easily obtained by using Lemma \ref{lem_commuting}-$(v),(vi)$.
\begin{theorem} \label{thm_commutingordinary}
    Let $X,Y \in \X_{M}(M)$. Let $\theta^{X}: D_{1} \rightarrow M$ and $\theta^{Y}: D_{2} \rightarrow M$ be the corresponding flows. Let $\bbD \subseteq (M \times \R) \times \R$ be the maximal commuting domain of $X$ and $Y$. 

    Then (\ref{eq_commuteordinary}) holds for all $((m,s),t) \in \bbD$, iff $[X,Y] = 0$.
\end{theorem}
Let us now go back to the graded setting. Let $\M$ be a graded manifold. Let $X,Y \in \X_{\M}(M)$ be two vector fields. Let $k := |X|$ and $\ell := |Y|$ be their degrees. Let $\theta^{X}: \D_{1} \rightarrow \M$ and $\theta^{Y}: \D_{2} \rightarrow \M$ be their flows. We can consider the canonical flip map 
\begin{equation}
    \sigma: (\M \times \R[-k]) \times \R[-\ell] \rightarrow (\M \times \R[-\ell]) \times \R[-k].
\end{equation}
The definition of maximal commuting domain of $X$ and $Y$ will now depend on their degrees. 
\begin{definice}
    A \textbf{maximal commuting domain} of $X$ and $Y$ is the open submanifold 
    \begin{equation}
    \sfD := ((\M \times \R[-k]) \times \R[-\ell])|_{\bbD},
    \end{equation}
    where the open subset $\bbD$ is defined as follows:
    \begin{enumerate}
        \item For $k = 0$ and $\ell = 0$, $\bbD \subseteq (M \times \R) \times \R$ is a maximal commuting domain of $X_{0}$ and $Y_{0}$.
        \item For $k = 0$ and $\ell \neq 0$, let $\bbD = D_{1} \times \{\ast\}$.
        \item For $k \neq 0$ and $\ell = 0$, let $\bbD = \ul{\sigma}^{-1}( D_{2} \times \{\ast\})$.
        \item For $k \neq 0$ and $\ell \neq 0$, let $\bbD = (M \times \{\ast\}) \times \{\ast\}$.
    \end{enumerate}
    In all cases, $\{ \ast \}$ denotes the underlying  one-point manifold of $\R[-k]$ whenever $k \neq 0$. Moreover, let us consider the open subset $\bbD' := \ul{\sigma}(\bbD)$ and the corresponding open submanifold.
    \begin{equation}
        \sfD' := ((\M \times \R[-\ell]) \times \R[-k])|_{\bbD'}.
    \end{equation}
\end{definice}
This allows us to define a notion of commuting flows of vector fields on graded manifolds.
\begin{definice}
    Let $X \in \X_{\M}(M)$ and $Y \in \X_{\M}(M)$. Let $\theta^{X}: \D_{1} \rightarrow \M$ and $\theta^{Y}: \D_{2} \rightarrow \M$ be the corresponding flows. We say that the flows $\theta^{X}$ and $\theta^{Y}$ \textbf{commute} on the maximal commuting domain $\sfD$, if the diagram
    \begin{equation} \label{eq_commutingdiagram}
        \begin{tikzcd}
            \sfD \arrow{r}{\theta^{X} \times \1_{\R[-\ell]}} \arrow{d}{\sigma} &[3em] \D_{2} \arrow{dd}{\theta^{Y}} \\
            \sfD' \arrow{d}{\theta^{Y} \times \1_{\R[-k]}} & \\[2em]
            \D_{1} \arrow{r}{\theta^{X}} & \M
        \end{tikzcd}
    \end{equation}
    commutes. We do not explicitly write the restrictions of the involved maps.
\end{definice}
We can finally state the main theorem of this section.
\begin{theorem} \label{thm_commuting}
    Let $X,Y \in \X_{\M}(M)$ be vector fields on $\M$. Let $\theta^{X}: \D_{1} \rightarrow \M$ and $\theta^{Y}: \D_{2} \rightarrow \M$ be the corresponding flows. 

    Then the flows $\theta^{X}$ and $\theta^{Y}$ commute on the maximal commuting domain $\sfD$, iff $[X,Y] = 0$. 
\end{theorem}
\begin{proof}
    The proof this theorem is in Appendix \ref{subsec_proofcommuting1} and in Appendix \ref{subsec_proofcommuting2}.
\end{proof}
\section{Flows of related vector fields} \label{sec_related}
Suppose that $\phi: \M \rightarrow \cN$ is any given graded smooth map. Let $X \in \X_{\M}(M)$ and $Y \in \X_{\cN}(N)$ be two vector fields of the same degree $k \in \Z$. Let $\theta: \D \rightarrow \M$ and $\theta': \D' \rightarrow \cN$ be their respective flows. We can ask if there is some natural relation of flows of vector fields, if they are $\phi$-related. Let us start by the following simple observation.
\begin{lemma} \label{lem_underlyingrelated}
    Let $k = 0$. Suppose that $X \sim_{\phi} Y$. Then $X_{0} \sim_{\ul{\phi}} Y_{0}$. 
\end{lemma}
\begin{proof}
    Write what $X \sim_{\phi} Y$ means when applied on a function on $\cN$, apply the body map on both sides and use the definition (\ref{eq_underlying}). 
\end{proof}

\begin{lemma} \label{lem_underlyingequivariant}
    Let $k = 0$. Suppose that $X \sim_{\phi} Y$. Then $(\ul{\phi} \times \1_{\R})(D) \subseteq D'$ and 
    \begin{equation} \label{eq_flowsrelatedunderlying}
        \ul{\phi}( \theta_{0}(m,t)) = \theta'_{0}( \ul{\phi}(m), t),
    \end{equation}
    for all $(m,t) \in D$. $\theta_{0}$ and $\theta'_{0}$ denote the flows of $X_{0}$ and $Y_{0}$, respectively. 
\end{lemma}
\begin{proof}
    Since $X_{0}$ and $Y_{0}$ are $\ul{\phi}$-related by Lemma \ref{lem_underlyingrelated}, this is a well-known standard statement. 
\end{proof}
It turns out that related vector fields have related flows and vice versa.
\begin{tvrz}
    Let $X \in \X_{\M}(M)$ and $Y \in \X_{\cN}(N)$ be two vector fields of degree $k \in \Z$. Let $\phi: \M \rightarrow \cN$ be a graded smooth map. Let $\theta: \D \rightarrow \M$ and $\theta': \D' \rightarrow \cN$ be their respective flows.

    Then $X \sim_{\phi} Y$, if and only if the following diagram commutes:
    \begin{equation} \label{eq_phiequivariant}
        \begin{tikzcd}
            \D \arrow{r}{\phi \times \1_{\R[-k]}} \arrow{d}{\theta} &[2em] \D' \arrow{d}{\theta'} \\
            \M \arrow{r}{\phi} & \cN
        \end{tikzcd}.
    \end{equation}  
    In other words, if we view $\theta$ and $\theta'$ as (kind of) actions of $(\R[-k],+)$, two vector fields are $\phi$-related, iff $\phi$ is $\R[-k]$-equivariant with respect to their flows.   
\end{tvrz}
\begin{proof}
    Let us only prove the $k = 0$ case. Suppose that $X \sim_{\phi} Y$. First, observe that
    \begin{equation}
        (1 \otimes \partial_{t}) \circ (\phi \times \1_{\R})^{\ast} \circ \theta'^{\ast} = (\phi \times \1_{\R})^{\ast} \circ (1 \otimes \partial_{t}) \circ \theta'^{\ast} = (\phi \times \1_{\R})^{\ast} \circ \theta'^{\ast} \circ Y.
    \end{equation}
    In other words, $1 \otimes \partial_{t}$ is $(\theta' \circ (\phi \times \1_{\R}))$-related to $Y$. Similarly, one has 
    \begin{equation}
        (1 \otimes \partial_{t}) \circ (\theta^{\ast} \circ \phi^{\ast}) = \theta^{\ast} \circ X \circ \phi^{\ast} = (\theta^{\ast} \circ \phi^{\ast}) \circ Y.
    \end{equation}
    Hence $1 \otimes \partial_{t}$ is also $(\phi \circ \theta)$-related to $Y$. Moreover, observe that $\theta' \circ (\phi \times \1_{\R})$ and $\phi \circ \theta$ have the same underlying maps by Lemma \ref{lem_underlyingequivariant} and their composition with $\ds_{0}: \M \rightarrow \D$ gives $\phi$. Hence by Lemma \ref{lem_thetathetaprimesameequationsoffsetgeneral}, they are equal. Hence (\ref{eq_phiequivariant}) commutes. To prove the converse, write the commutativity of (\ref{eq_phiequivariant}) in terms of pullbacks, act on both sides with $1 \otimes \partial_{t}$ and pull back by the zero section $\ds_{0}$. 

    The proof of the $k \neq 0$ case is the same, except one acts on everything by $\ds_{0}^{\ast} \circ (1 \otimes \partial_{\tau})^{r}$ and compares the results for all $r \in \N_{0}$, see the discussion above (\ref{eq_even_component_functions}). 
\end{proof}
\printbibliography
\appendix
\section{Proof of theorem: degree zero} \label{app_proof0}
This is a section dedicated to the proof of Theorem \ref{thm_funamental_extraordinary} for a degree zero vector field $X \in \X_{\M}(M)$. This is by far the most complicated case. It is also the most interesting, since the flows of degree zero vector fields are the only ones that generate (possibly local) graded diffeomorphisms of $\M$. We will divide the proof into the several steps:
\begin{enumerate}[(1)]
    \item Set up a necessary notation to work with functions on graded domains.
    \item Show that there is a unique solution to a certain system of differential equations for maps between graded domains, with certain initial conditions. 
    \item Use this local statement to construct a graded smooth map $\theta: \D \rightarrow \M$ on some flow domain $\D$ of degree zero, satisfying the equations
    \begin{equation} \label{eq_thetaequationmainfundamental}
        (1 \otimes \partial_{t}) \circ \theta^{\ast} = \theta^{\ast} \circ X, \; \; \theta \circ (\1_{\M},0_{\M}) = \1_{\M}.
    \end{equation}
    
    \item Show that if $\theta': \D' \rightarrow M$ is other such map, then $\theta|_{D \cap D'} = \theta'|_{D \cap D'}$.
    \item Show that $\theta$ can be defined on a union of all flow domains where the solution of (\ref{eq_thetaequationmainfundamental}) exists. This union is in fact equal to the maximal flow domain $D_{0}$ of the underlying vector field $X_{0}$.
    \item Argue that $\theta$ is indeed a flow, that is it fits into the second diagram in (\ref{eq_flowdiagrams}).
\end{enumerate}

\subsection{Setting the notation}\label{subsection_setting_the_notation}
Let $(n_{j})_{j \in \Z}$ be a sequence of non-negative integers with $\sum_{j \in \Z} n_{j} < \infty$. Suppose $U \in \Op(\R^{n_{0}})$ and let $U^{(n_{j})}$ be a graded domain with a structure sheaf $\C^{\infty}_{(n_{j})}$. Let $( x^{i})_{i=1}^{n_{0}}$ denote both the coordinates on $U$ and degree zero coordinates on $U^{(n_{j})}$, and let $(\xi^{\mu})_{\mu=1}^{n_{\ast}}$ be coordinates of a non-zero degree on $U^{(n_{j})}$. For each $k \in \Z$, let us define the following subset:
\begin{equation}
    \ol{\N}{}^{n_{\ast}}_{k} := \{ \fp = (p_{1}, \dots, p_{n_{\ast}}) \in (\N_{0})^{n_{\ast}} \mid p_{\mu} |\xi^{\mu}| = k \text{ and } p_{\mu} \in \{0,1 \} \text{ for odd } |\xi^{\mu}|\}.
\end{equation}
This is precisely the set of $n_{\ast}$-indices ensuring that the product $\xi^{\fp} := (\xi^{1})^{p_{1}} \cdots (\xi^{n_{\ast}})^{p_{n_{\ast}}}$ is a nonzero function on $U^{(n_j)}$ of degree $k$. We define $w(\fp) := \sum_{\mu=1}^{n_{\ast}} p_{\mu}$ to be the overall \textit{weight} of $\fp$. Let
\begin{equation} \label{eq_fnula}
    \fnula := (0, \dots, 0),
\end{equation}
denote the $n_{\ast}$-index consisting solely of zeros. Note that $\fnula \in \ol{\N}{}^{n_{\ast}}_{0}$ and $w(\fnula) = 0$. Moreover, for each $\mu \in \{1,\dots,n_{\ast} \}$, let us henceforth write $|\mu| := |\xi^{\mu}|$ and define $\fp(\mu) \in \ol{\N}{}^{n_{\ast}}_{|\mu|}$ by $[\fp(\mu)]_{\nu} = \delta_{\mu \nu}$. This is the $n_{\ast}$-index consisting of zeros, except for $1$ on the $\mu$-th place. Note that $w(\fp(\mu)) = 1$.

A general function $f \in \C^{\infty}_{(n_{j})}(U)$ of degree $|f|$ is then a formal power series
\begin{equation} \label{eq_fasformalpowerseries}
    f = \sum_{\fp \in \ol{\N}{}^{n_{\ast}}_{|f|}} f_{\fp} \xi^{\fp}, 
\end{equation}
where $f_{\fp} = f_{\fp}(x^{1},\dots,x^{n_{0}})$ are ordinary smooth functions on $U$. Next, let us also use the symbol $\ul{\N}^{n_{0}}$ for the set of all $n_{0}$-indices valued in non-negative integers. We write $\fI = (I_{1},\dots,I_{n_{0}}) \in \ul{\N}^{n_{0}}$ and define $w(\fI) := \sum_{i=1}^{n_{0}} I_{i}$. 

Next, suppose $X \in \X_{(n_{j})}(U)$ is a degree zero vector field on $U^{(n_{j})}$. We can decompose it with respect to the coordinate vector fields as 
\begin{equation}
    X = X^{i} \partial_{i} + X^{\mu} \partial_{\mu},
\end{equation}
where we will always write $\partial_{i} := \frac{\partial}{\partial x^{i}}$ and $\partial_{\mu} = \frac{\partial}{\partial \xi^{\mu}}$. Note that $|X^{i}| = 0$ and $|X^{\mu}| = |\mu|$. We can write the component functions as formal power series
\begin{equation}
    X^{i} = \sum_{\fp \in \ol{\N}{}^{n_{\ast}}_{0}} X^{i}_{\fp} \xi^{\fp}, \; \; X^{\mu} = \sum_{\fp \in \ol{\N}{}^{n_{\ast}}_{|\mu|}} X^{\mu}_{\fp} \xi^{\fp}.
\end{equation}
for each $i \in \{1,\dots,n_{0}\}$ and $\mu \in \{1,\dots,n_{\ast}\}$. 

Finally, suppose we are given $U,V \subseteq \R^{n_{0}}$ and an open interval $I$ containing a given point $t_{0} \in \R$. Let $\theta_{0}: V \times I \rightarrow U$ be a given smooth map. Let us now parametrize a general graded smooth map $\theta: V^{(n_{j})} \times I \rightarrow U^{(n_{j})}$ satisfying $\ul{\theta} = \theta_{0}$. Every graded smooth map into a graded domain is fully determined by pullbacks of coordinate functions. Let us write those as
\begin{align}
\theta^{\ast}(x^{i}) =: & \ x^{i}_{\ast} \equiv \sum_{\fp \in \ol{\N}{}^{n_{\ast}}_{0}} \theta^{i}_{\fp} \xi^{\fp}, \\
\theta^{\ast}(\xi^{\mu})=: & \ \xi^{\mu}_{\ast}  \equiv \sum_{\fp' \in \ol{\N}{}^{n_{\ast}}_{|\mu|}} \theta^{\mu}_{\fp'}\xi^{\fp'}, 
\end{align}
where on the right hand side, we view $\xi$'s as coordinates on the product $V^{(n_{j})} \times I$, and $\theta^{i}_{\fp} = \theta^{i}_{\fp}(x,t)$ and $\theta^{\mu}_{\fp'} = \theta^{\mu}_{\fp'}(x,t)$ are ordinary smooth functions on $V \times I$. The assumption $\ul{\theta} = \theta_{0}$ means that 
\begin{equation}
    \theta^{i}_{\fnula} = x^{i} \circ \theta_{0},
\end{equation}
for each $i \in \{1, \dots, n_{0}\}$. Let us also write $\ol{x}{}^{i}_{\ast}$ for the ``purely graded'' part of $x^{i}_{\ast}$, that is 
\begin{equation}
    \ol{x}{}^{i}_{\ast} = x^{i}_{\ast} - x^{i} \circ \theta_{0}. 
\end{equation}
Finally, suppose that $f \in \C^{\infty}_{(n_{j})}(U)$ is a general function on $U^{(n_{j})}$. How does the pullback $\theta^{\ast}(f)$ look explicitly? Write $f$ as in (\ref{eq_fasformalpowerseries}). Then one can write 
\begin{equation}
    \theta^{\ast}(f) = \sum_{\fp \in \ol{\N}{}^{n_{\ast}}_{|f|}} [\theta^{\ast}(f)]_{\fp} \xi^{\fp}.
\end{equation}
The coefficient functions $[\theta^{\ast}(f)]_{\fp}$ on $V \times I$ are for each $\fp \in \ol{\N}{}^{n_{\ast}}_{|f|}$ given by the formula
\begin{equation} \label{eq_thetastfpthcomponent}
    [\theta^{\ast}(f)]_{\fp} = \sum_{\substack{\fq \in \ol{\N}{}^{n_{\ast}}_{|f|}\\ \fq \leq \fp}} \epsilon_{\fp}^{(\fp - \fq, \fq)} \sum_{\substack{\fr \in \ol{\N}{}^{n_{\ast}}_{|f|}\\ { w(\fr) \leq w(\fq)}}} [\ol{\theta}{}^{\ast}(f_{\fr})]_{\fp - \fq} [\xi^{\fr}_{\ast}]_{\fq},
\end{equation}
where $\epsilon_{\fp}^{(\fp-\fq,\fq)}$ is the sign obtained from the expression $\xi^{\fp - \fq} \cdot \xi^{\fq} = \epsilon_{\fp}^{(\fp-\fq,\fq)} \xi^{\fp}$ and 
\begin{equation} \label{eq_oltheta}
    [\ol{\theta}{}^{\ast}(f_{\fr})]_{\fp - \fq} := \sum_{\substack{\fI \in \ul{\N}^{n_{0}}\\w(\fI) \leq w(\fp - \fq)}} \frac{1}{\fI!} ([\partial_{\fI} f_{\fr}] \circ \theta_{0}) \cdot [\ol{x}^{\fI}_{\ast}]_{\fp - \fq}.
\end{equation}
We use the shorthands $\xi_{\ast}^{\fr} = (\xi^{1}_{\ast})^{r_{1}} \cdots (\xi^{n_{\ast}}_{\ast})^{r_{n_{\ast}}}$, $\ol{x}^{\fI}_{\ast} = (\ol{x}^{1}_{\ast})^{I_{1}} \cdots (\ol{x}^{n_{0}}_{\ast})^{I_{n_{0}}}$, $\partial_{\fI} = (\partial_{1})^{I_{1}} \cdots (\partial_{n_{0}})^{I_{n_{0}}}$ and $\fI! = (i_{1}!) \cdots (i_{n_{0}}!)$. This is a bit brutal, but it is just how the pullback of a general function $f$ can be expressed in terms of its coefficient functions $f_{\fr}$ and pulbacks $\ol{x}^{i}_{\ast}$ and $\xi^{\mu}_{\ast}$ of coordinate functions. For details, see \S 3.2 in \cite{Vysoky:2022gm}.

\subsection{Pivotal system and its solution} \label{subsec_pivotal}
Let $U,V \in \Op(\R^{n_{0}})$ and $I \subseteq \R$ be an open interval containing $t_{0}$. Let $X \in \X_{(n_{j})}(U)$ be a given degree zero vector field on $U^{(n_{j})}$. Let $\phi: V^{(n_{j})} \rightarrow U^{(n_{j})}$ be a given graded smooth map. We can parametrize the map $\phi$ as 
    \begin{equation} \label{eq_phimapparametrization}
        \phi^{\ast}(x^{i}) = \sum_{\fp \in \ol{\N}{}^{n_{\ast}}_{0}} \phi^{i}_{\fp} \xi^{\fp}, \; \; \phi^{\ast}(\xi^{\mu}) = \sum_{\fp' \in \ol{\N}{}^{n_{\ast}}_{|\mu|}} \phi^{\mu}_{\fp'} \xi^{\fp'}.
    \end{equation}
Note that $\phi^{i}_{\fnula} = x^{i} \circ \ul{\phi}$. Let $\theta_{0}: V \times I \rightarrow U$ be a given smooth map satisfying the equations
\begin{equation} \label{eq_theta0assumption}
    (1 \otimes \partial_{t}) \circ \theta_{0}^{\ast} = \theta_{0}^{\ast} \circ X_{0}, \; \; \theta_{0}(x,t_{0}) = \ul{\phi}(x),
\end{equation}
for all $x \in V$. Let $\ds_{t_{0}}: V^{(n_{j})} \rightarrow V^{(n_{j})} \times I$ be the graded smooth map given by
\begin{equation}
\ds_{t_{0}}^{\ast}( x^{i}) := x^{i}, \;\; \ds_{t_{0}}^{\ast}( \xi^{\mu}) := \xi^{\mu}, \; \; \ds_{t_{0}}^{\ast}(t) := t_{0}.
\end{equation}
In other words, it is just a \textbf{section valued at $t_{0}$}. Its underlying smooth map is $\ul{s}_{t_{0}}(x) = (x,t_{0})$. The following statement is the essential computational result pivotal for the main proof. 
\begin{tvrz} \label{tvrz_pivotal}
        There exists a unique graded smooth map $\theta: V^{(n_{j})} \times I \rightarrow U^{(n_{j})}$ satisfying 
    \begin{equation} \label{eq_pivotalODE}
        (1 \otimes \partial_{t}) \circ \theta^{\ast} = \theta^{\ast} \circ X, \; \; \ds_{t_{0}}^{\ast} \circ \theta^{\ast} = \phi^{\ast},
    \end{equation}
    such that $\ul{\theta} = \theta_{0}$. 
\end{tvrz}
\begin{proof}
    It suffices to verify (\ref{eq_pivotalODE}) evaluated on the coordinate functions and compare the component functions of the resulting formal power series. This leads to the system of differential equations
    \begin{align} 
        \partial_{t} \theta^{i}_{\fp}(x,t) = & \ [\theta^{\ast}(X^{i})]_{\fp}(x,t), \label{eq_pivotalODE2} \\
        \partial_{t} \theta^{\mu}_{\fp'}(x,t) = & \ [\theta^{\ast}(X^{\mu})]_{\fp'}(x,t). \label{eq_pivotalODE3}
    \end{align}
    for functions on $V \times I$, with the ``initial conditions'' written using the parametrization (\ref{eq_phimapparametrization}) as 
    \begin{equation} \label{eq_pivotalODE4}
        \theta^{i}_{\fp}(x,t_{0}) = \phi^{i}_{\fp}(x), \; \;\theta^{\mu}_{\fp'}(x,t_{0}) = \phi^{\mu}_{\fp'}(x),
    \end{equation}
    for all $i \in \{1,\dots,n_{0}\}$, $\fp \in \ol{\N}{}^{n_{\ast}}_{0}$, $\mu \in \{1,\dots,n_{\ast}\}$ and $\fp' \in \ol{\N}{}^{n_{\ast}}_{|\mu|}$. 

    Let us now prove that this system can be solved inductively on $w \in \N_{0}$, where $w$ is the weight of the $n_{\ast}$-index $\fp$ or $\fp'$ in (\ref{eq_pivotalODE2}, \ref{eq_pivotalODE3}) and (\ref{eq_pivotalODE4}). 

    \begin{enumerate}[(1)]
        \item $w = 0$: This can happen only for $i \in \{ 1, \dots, n_{0}\}$ and $\fp = \fnula$, where (\ref{eq_pivotalODE2}) turns into 
        \begin{equation} \label{eq_ODE5}
            \partial_{t} \theta^{i}_{\fnula}(x,t) = X^{i}_{\fnula}( \theta_{0}(x,t)),
        \end{equation}
        and the (\ref{eq_pivotalODE4}) into $\theta^{i}_{\fnula}(x,t_{0}) = \phi^{i}_{\fnula}(x)$. Since we require $\ul{\theta} = \theta_{0}$, the unique choice is $\theta^{i}_{\fnula} = x^{i} \circ \theta_{0}$ and the equation (\ref{eq_ODE5}) together with the initial condition turns into the assumption (\ref{eq_theta0assumption}) about $\theta_{0}$. This finishes the $w = 0$ step. 

        \item $w = 1$: There are no elements of $\ol{\N}{}^{n_{\ast}}_{0}$ of weight $1$, hence only (\ref{eq_pivotalODE3}) is relevant in this case. For each $j \in \Z\ssm \{0\}$, let us introduce the subset
        \begin{equation}
            P_{j} = \{ \nu \in \{1, \dots, n_{\ast} \} \mid |\nu| = j\}.
        \end{equation}
        Note that for any $\nu \in P_{j}$, one has $\fp(\nu) \in \ol{\N}{}^{n_{\ast}}_{j}$. See under (\ref{eq_fnula}) for the definition. Using this notation, the only non-trivial equations with $w = 1$ coming from (\ref{eq_pivotalODE3}) and for a given $\mu \in \{1, \dots, n_{\ast} \}$ are for $\fp' = \fp(\nu)$ where $\nu \in P_{|\mu|}$. 
        
        We have to analyze the expression (\ref{eq_thetastfpthcomponent}) for $[\theta^{\ast}(X^{\mu})]_{\fp(\nu)}$. Recall that $|X^{\mu}| = |\mu|$, so in the sum over $\fq$, only $\fq = \fp(\nu)$ contributes, and in the sum over $\fr$, we only sum over $\fr = \fp(\lambda)$ with $\lambda \in P_{|\mu|}$. Hence
        \begin{equation}
            [\theta^{\ast}(X^{\mu})]_{\fp(\nu)} = \sum_{\lambda \in P_{|\mu|}} [\ol{\theta}{}^{\ast}(X^{\mu}_{\fp(\lambda)})]_{\fnula} [\xi^{\fp(\lambda)}_{\ast}]_{\fp(\nu)}.
        \end{equation}
        Now, since only $\fI = (0,\dots,0)$ contributes to the sum (\ref{eq_oltheta}) and in this case $[\ol{x}^{\fI}_{\ast}]_{\fnula} = 1$, we find
        \begin{equation}
            [\ol{\theta}^{\ast}(X^{\mu}_{\fp(\lambda)})]_{\fnula} = X^{\mu}_{\fp(\lambda)} \circ \theta_{0}.
        \end{equation}
        Note that $[\xi_{\ast}^{\fp(\lambda)}]_{\fp(\nu)} = \theta^{\lambda}_{\fp(\nu)}$. By plugging into (\ref{eq_pivotalODE3}), for each $\mu \in \{1,\dots,n_{\ast}\}$ and $\nu \in P_{|\mu|}$ we thus find the ordinary differential equation
        \begin{equation} \label{eq_pivotalODE7}
            \partial_{t} \theta^{\mu}_{\fp(\nu)}(x,t) = \sum_{\lambda \in P_{|\mu|}} X^{\mu}_{\fp(\lambda)}(\theta_{0}(x,t)) \theta^{\lambda}_{\fp(\nu)}(x,t).
        \end{equation}
        Now, for each $j \in \Z \ssm \{0\}$ with $n_{j} > 0$, form the $n_{j} \times n_{j}$ matrix $\fA_{j}$ valued in smooth functions on $V \times I$, labeled by $\mu,\lambda \in P_{j}$:
        \begin{equation}
            [\fA_{j}]^{\mu}{}_{\lambda} = X^{\mu}_{\fp(\lambda)} \circ \theta_{0}. 
        \end{equation}
        For every $j \in \Z \ssm \{0\}$ with $n_{j} > 0$, we can thus form the fundamental solution $\Phi_{j}$ of the system of ordinary differential equations given by $\fA_{j}$, that is the $n_{j} \times n_{j}$ matrix of smooth functions on $V \times I$ satisfying 
        \begin{equation} \label{eq_fundamentalPhij}
            \partial_{t} \Phi_{j}(x,t) = \fA_{j}(x,t) \cdot \Phi_{j}(x,t), \; \; \Phi_{j}(x,t_{0}) = \f1_{n_{j}}.
        \end{equation}
        The matrix $\Phi_{j}$ always exists, it is unique and $\Phi_{j}(x,t)$ is invertible for all $(x,t) \in V \times I$. This is a standard well-known statement, see e.g. \cite{coddington1956theory}. The initial conditions for the functions $\theta^{\mu}_{\fp(\nu)}$ take the form
        \begin{equation} \label{eq_pivotalODE8}
            \theta^{\mu}_{\fp(\nu)}(x,t_{0}) = \phi_{\fp(\nu)}^{\mu}(x),
        \end{equation}
        for all $x \in V$. For a fixed $\mu \in \{1, \dots, n_{\ast} \}$ and $\nu \in P_{|\mu|}$, the unique solution to (\ref{eq_pivotalODE7}) and (\ref{eq_pivotalODE8}) can be now obtained from the matrices $\Phi_{j}$ by 
        \begin{equation}
            \theta^{\mu}_{\fp(\nu)}(x,t) := \sum_{\lambda \in P_{|\mu|}} [\Phi_{|\mu|}]^{\mu}{}_{\lambda}(x,t) \phi^{\lambda}_{\fp(\nu)}(x),
        \end{equation}
        for all $(x,t) \in V \times I$. We leave the easy verification to the reader. 
        \item $w > 1$: We now assume that we have solutions for the system (\ref{eq_pivotalODE2}, \ref{eq_pivotalODE3}) with the initial conditions (\ref{eq_pivotalODE4}) for all multiindices of weight strictly lower then $w$. 

        Let $i \in \{1, \dots, n_{0} \}$ and $\fp \in \ol{\N}{}^{n_{\ast}}_{0}$, $w(\fp) = w$. The idea is to write the right-hand side of (\ref{eq_pivotalODE2}) as a sum of two terms. The first one will contain only terms depending on $\fp$, the second one will depend only on functions $\theta^{j}_{\fq}$ and $\theta^{\nu}_{\fq'}$ where $w(\fq)$, $w(\fq') < w$. Since we are solving the system inductively on $w$, we can then plug into the second part and obtain a system of ordinary differential equations for functions depending only on $\fp$. Recall that 
        \begin{align}
            [\theta^{\ast}(X^{i})]_{\fp} = & \ \sum_{\substack{\fq \in \ol{\N}{}^{n_{\ast}}_{0}\\ \fq \leq \fp}} \epsilon_{\fp}^{(\fp - \fq, \fq)} \sum_{\substack{\fr \in \ol{\N}{}^{n_{\ast}}_{0}\\ { w(\fr) \leq w(\fq)}}} [\ol{\theta}{}^{\ast}(X^{i}_{\fr})]_{\fp - \fq} [\xi^{\fr}_{\ast}]_{\fq}, \\ \label{eq_equationODE9}
            [\ol{\theta}{}^{\ast}(X^{i}_{\fr})]_{\fp - \fq} = & \  \sum_{\substack{\fI \in \ul{\N}^{n_{0}}\\w(\fI) \leq w(\fp - \fq)}} \frac{1}{\fI!} ([\partial_{\fI} X^{i}_{\fr}] \circ \theta_{0})  [\ol{x}^{\fI}_{\ast}]_{\fp - \fq}.
        \end{align}
        The top weight $w$ multiindices appear possibly only for $\fq \in \{ \fnula, \fp \}$. 

        \begin{enumerate}[(a)] \item First, suppose that $\fq = \fnula$. Then the summand in the second sum is $\fr = \fnula$, so the whole $\fq = \fnula$ term is just
        \begin{equation}
            [\ol{\theta}{}^{\ast}(X^{i}_{\fnula})]_{\fp} = \sum_{\substack{\fI \in \ul{\N}^{n_{0}}\\w(\fI) \leq w}} \frac{1}{\fI!} ([\partial_{\fI} X^{i}_{\fnula}] \circ \theta_{0}) \cdot [\ol{x}^{\fI}_{\ast}]_{\fp}.
        \end{equation}
        But note that $\ol{x}^{\fI}_{\ast}$ is a $w(\fI)$-fold product of formal power series, so $[\ol{x}^{\fI}_{\ast}]_{\fp}$ has the form
        \begin{equation}
            [\ol{x}^{\fI}_{\ast}]_{\fp} = \sum_{\vec{\fm} \in \ol{\N}{}^{(w(\fI))}_{\fp}} \epsilon^{\vec{\fm}}_{\fp} \underbrace{ [\ol{x}^{1}_{\ast}]_{\fm_{1}} \cdots [\ol{x}^{1}_{\ast}]_{\fm_{I_{1}}} }_{I_{1} \text{ terms}}\cdots \underbrace{[\ol{x}{}^{n_{0}}_{\ast}]_{\fm_{I_{1} + \dots + I_{n_{0}-1} + 1}} \cdots [\ol{x}{}^{n_{0}}_{\ast}]_{\fm_{w(\fI)}} }_{I_{n_{0}} \text{ terms}},
        \end{equation}
        where we sum over the following set of $w(\fI)$-tuples of $n_{\ast}$-indices of non-negative integers
        \begin{equation}
            \ol{\N}{}^{(w(\fI))}_{\fp} = \{ \vec{\fm} = (\fm_{1}, \dots, \fm_{w(\fI)}) \mid \fm_{1} + \cdots + \fm_{w(\fI)} = \fp \},
        \end{equation}
        and $\epsilon^{\vec{\fm}}_{\fp} \in \{-1,1\}$ is the unique sign obtained be reordering $\xi^{\fm_{1}} \cdots \xi^{\fm_{w(\fI)}}$ to $\xi^{\fp}$. 
        
        Now, one of the $n_{\ast}$-indices in $\vec{\fm}$ can have a top weight $w$, only if it is equal to $\fp$. Since $[\ol{x}^{j}_{\ast}]_{\fnula} = 0$, this can only happen if $w(\fI) = 1$. In other words, we must have $\fI = \fI(k)$ for some $k \in \{1, \dots, n_{0}\}$, where $(\fI(k))_{j} := \delta_{kj}$.  Since $[\ol{x}^{k}_{\ast}]_{\fp} = \theta^{k}_{\fp}$, we find that 
        \begin{equation}
            [\ol{\theta}{}^{\ast}(X^{i}_{\fnula})]_{\fp} = \sum_{k=1}^{n_{0}} ([\partial_{k} X^{i}_{\fnula}] \circ \theta_{0}) \theta^{k}_{\fp} + \text{lower weight terms}.
        \end{equation}
        \item Second, suppose that $\fq = \fp$. The corresponding summand of (\ref{eq_equationODE9}) simplifies to the sum
        \begin{equation}
           \sum_{\substack{\fr \in \ol{\N}{}^{n_{\ast}}_{0}\\ w(\fr) \leq w}} (X^{i}_{\fr} \circ \theta_{0}) [\xi^{\fr}_{\ast}]_{\fp}.
        \end{equation}
        Using the same argument as in (a), a weight $w$ component of $\theta$ can appear in $[\xi^{\fr}_{\ast}]_{\fp}$ only for $w(\fr) = 1$. But there is no such $\fr$ in $\ol{\N}{}^{n_{\ast}}_{0}$.
        \end{enumerate}
        It follows that (\ref{eq_pivotalODE2}) can be for each $i \in \{1, \dots, n_{0}\}$ and $\fp \in \ol{\N}{}^{n_{\ast}}_{0}$ with $w(\fp) = w$ written as
        \begin{equation} \label{eq_pivotalODE10}
            \partial_{t} \theta^{i}_{\fp}(x,t) = \sum_{k=1}^{n_{0}} ([\partial_{k} X^{i}_{\fnula}] \circ \theta_{0}) \theta^{k}_{\fp}(x,t) + y^{i}_{\fp}(x,t),
        \end{equation}
        where $y^{i}_{\fp}(x,t)$ depend on $\theta^{j}_{\fq}$ and $\theta^{\nu}_{\fq'}$ for $w(\fq), w(\fq') < w$. Since we are solving the system iteratively, we can assume that these are $y^{i}_{\fp}(x,t)$ are \textit{known functions} on $V \times I$. Let $\fA_{0}$ be the $n_{0} \times n_{0}$ matrix of smooth functions on $V \times I$ defined for each $i,k \in \{1,\dots,n_{0}\}$ by
        \begin{equation}
            [\fA_{0}]^{i}{}_{k} = [\partial_{k}X^{i}_{\fnula}] \circ \theta_{0}.
        \end{equation}
        Let $\Phi_{0}$  be the fundamental solution of the system of ordinary differential equations given by $\fA_{0}$, that is the $n_{0} \times n_{0}$ matrix of smooth functions on $V \times I$ satisfying
        \begin{equation}
            \partial_{t} \Phi_{0}(x,t) = \fA_{0}(x,t) \cdot \Phi_{0}(x,t), \; \; \Phi_{0}(x,t_{0}) = \f1_{n_{0}}.
        \end{equation}
        Again, $\Phi_{0}(x,t)$ is invertible for all $(x,t) \in V \times I$. We look for a solution to (\ref{eq_pivotalODE10}) satisfying the initial condition (\ref{eq_pivotalODE4}). We propose it in the form
        \begin{equation}
        \theta^{i}_{\fp}(x,t) := \sum_{k=1}^{n_{0}} [\Phi_{0}(x,t)]^{i}{}_{k} \{ \phi^{k}_{\fp}(x) + c^{k}_{\fp}(x,t) \},
        \end{equation}
        where $c^{k}_{\fp}$ are yet undetermined smooth functions on $V \times I$. By plugging into (\ref{eq_pivotalODE10}) and (\ref{eq_pivotalODE4}), we obtain the conditions
        \begin{equation}
            \partial_{t} c^{k}_{\fp}(x,t) = \sum_{\ell=1}^{n_{0}} y^{\ell}_{\fp}(x,t)  [\Phi_{0}^{-1}(x,t)]^{k}{}_{\ell}, \; \; c^{k}_{\fp}(x,t_{0}) = 0.
        \end{equation}
        This is easily and uniquely solved by smooth functions 
        \begin{equation}
            c^{k}_{\fp}(x,t) := \sum_{\ell=1}^{n_{0}} \int_{t_{0}}^{t} y^{\ell}_{\fp}(x,s)  [\Phi_{0}^{-1}(x,s)]^{k}{}_{\ell} \; \dr{s}. 
        \end{equation}
        We conclude that for any $i \in \{1, \dots, n_{0} \}$ and any $\fp \in \ol{\N}{}^{n_{\ast}}_{0}$ with $w(\fp) = w$, we are able to find unique smooth functions $\theta^{i}_{\fp}$ on $V \times I$ satisfying (\ref{eq_pivotalODE2}) together with the initial conditions (\ref{eq_pivotalODE4}). 

        Next, for any $\mu \in \{1, \dots, n_{\ast} \}$ and any $\fp' \in \ol{\N}{}^{n_{\ast}}_{|\mu|}$ with $w(\fp') = w$, one can follow the same line of arguments to show that (\ref{eq_pivotalODE3}) can be rewritten as 
        \begin{equation}
            \partial_{t} \theta^{\mu}_{\fp'}(x,t) = \sum_{\lambda \in P_{|\mu|}} (\theta^{\mu}_{\fp(\lambda)} \circ \theta_{0}) \theta^{\lambda}_{\fp'}(x,t) + y^{\mu}_{\fp'}(x,t),
        \end{equation}
        where $y^{\mu}_{\fp^\prime}(x,t)$ depend on $\theta^{j}_{\fq}$ and $\theta^{\nu}_{\fq'}$ for $w(\fq), w(\fq') < w$. We can thus view those as known functions on $V \times I$. It follows that the unique solution satisfying (\ref{eq_pivotalODE4}) can be written as 
        \begin{equation}
            \theta^{\mu}_{\fp'}(x,t) = \sum_{\lambda \in P_{|\mu|}} [\Phi_{|\mu|}(x,t)]^{\mu}{}_{\lambda} \{ \phi^{\lambda}_{\fp'}(x) + c^{\lambda}_{\fp'}(x,t) \},
        \end{equation}
        where $\Phi_{j}$ are the fundamental solutions (\ref{eq_fundamentalPhij}) obtained for each $j \in \Z \ssm \{0\}$ with $n_{j} > 0$, and $c^{\lambda}_{\fp'}(x,t)$ are for each $\lambda \in P_{|\mu|}$ defined by the integral
        \begin{equation}
            c^{\lambda}_{\fp'}(x,t) := \sum_{\nu \in P_{|\mu|}} \int_{t_{0}}^{t} y^{\nu}_{\fp'}(x,s) [\Phi_{|\mu|}^{-1}(x,s)]^{\lambda}{}_{\nu} \; \dr{s}. 
        \end{equation}
        This finishes the iteration step. 
    \end{enumerate}
    It is important that for $w > 0$, we always solve a system of linear ordinary differential equations. Such systems admit global solutions, which is absolutely vital for the iterative argument to be valid. This concludes the proof.
\end{proof}
\begin{rem}
Note that the difficult bit is to find the fundamental solutions $\Phi_{j}$ for any $j \in \Z$ with $n_{j} > 0$. Since $\sum_{j \in \Z} n_{j} < \infty$, there is finitely many of those. 
\end{rem}
\subsection{Existence of the general solution} \label{subsec_genex}
Let $X \in \X_{\M}(M)$ be a degree zero vector field on a general graded manifold $\M$. Let $X_{0} \in \X_{M}(M)$ be its underlying vector field, see Lemma \ref{lem_underlying}. Let $\theta_{0}: D_{0} \rightarrow M$ be the flow of $X_{0}$ obtained by means of Theorem \ref{thm_fundamentalordinary}. In this subsection, we will construct a graded smooth map $\theta: \D \rightarrow \M$ satisfying the correct differential equations, on \textit{some} flow domain $D \subseteq M \times \R$. In view of Proposition \ref{tvrz_underlyingflow}, we must ensure that $\ul{\theta}$ and $\theta_{0}$ agree on the intersection of their domains. We will show that $D$ can be chosen to satisfy $D \subseteq D_{0}$, so we will require $\ul{\theta} = \theta_{0}|_{D}$.  

\begin{lemma} \label{lem_thetalocal}
    Let $(m,t_{0}) \in D_{0}$ be arbitrary. Fix local charts $(V,\psi)$  and $(U,\varphi)$ for $\M$, such that $V \in \Op_{m}(M)$. Fix any open interval $I \in \Op_{t_{0}}(\R)$, such that $V \times I \subseteq D_{0}$ and $\theta_{0}(V \times I) \subseteq U$. Let $\phi: \M|_{V} \rightarrow \M|_{U}$ be any given graded smooth map.

    Then there exists a unique graded smooth map $\theta: \M|_{V} \times I \rightarrow \M|_{U}$ satisfying the equations
    \begin{equation} \label{eq_pivotalODEUVlocal}
        (1 \otimes \partial_{t}) \circ \theta^{\ast} = \theta^{\ast} \circ X|_{U}, \; \; \theta \circ \ds_{t_{0}}^{V} = \phi,
    \end{equation}
    such that $\ul{\theta} = \theta_{0}|_{V \times I}$, where $\ds_{t_{0}}^{V}: \M|_{V} \rightarrow \M|_{V} \times I$ is the section valued at $t_{0} \in I$.
\end{lemma}
\begin{proof}
Use local charts to work with graded domains. Then utilize the results of the previous subsection, namely Proposition \ref{tvrz_pivotal}.
\end{proof}

\begin{tvrz} \label{tvrz_therexistsflow}
    Let $X \in \X_{\M}(M)$ be an arbitrary degree zero vector field. Then there exists a flow domain $\D$ on $\M$ of degree $0$, and a graded smooth map $\theta: \D \rightarrow \M$ satisfying
    \begin{equation} \label{eq_thetaquationalmostfinal}
        (1 \otimes \partial_{t}) \circ \theta^{\ast} = \theta^{\ast} \circ X, \; \; \theta \circ \ds_{0} = \1_{\M},
    \end{equation}
    where $\ds_{0} := (\1_{\M},0_{\M}): \M \rightarrow \D$ is the zero section. One has $D \subseteq D_{0}$, where $D_{0}$ is the flow domain for the flow $\theta_{0}$ of $X_{0}$, and $\ul{\theta} = \theta_{0}|_{D}$.
\end{tvrz}
\begin{proof}
    For each $m \in M$, one can find a local chart $(U,\varphi)$, where $U \in \Op_{m}(M)$. We have $\theta_{0}(m,0) = m$, so there is $V \in \Op_{m}(M)$ and an open interval $I$ containing zero, such that $\theta_{0}(V \times I) \subseteq U$. Note that necessarily $V \subseteq U$, since $\theta_{0}(v,0) = v$ for all $v \in V$. We can thus use $(V,\varphi|_{V})$ as a local chart for $\M$ over $V$.

    In this way, we may find open covers $\{ U_{\alpha} \}_{\alpha \in S}$ and $\{ V_{\alpha} \}_{\alpha \in S}$ of $M$, such that for each $\alpha \in S$, we have $V_{\alpha} \in \Op(U_{\alpha})$ and an open interval $I_{\alpha} \in \Op_{0}(\R)$ satisfying $V_{\alpha} \times I_{\alpha} \subseteq D_{0}$ and $\theta_{0}(V_{\alpha} \times I_{\alpha}) \subseteq U_{\alpha}$. For each $\alpha \in S$, use Lemma \ref{lem_thetalocal} to construct $\theta'_{\alpha}: \M|_{V_{\alpha}} \times I_{\alpha} \rightarrow \M|_{U_{\alpha}}$ satisfying 
    \begin{equation}
        (1 \otimes \partial_{t}) \circ \theta'^{\ast}_{\alpha} = \theta'^{\ast}_{\alpha} \circ X|_{U_{\alpha}}, \; \; \theta'_{\alpha} \circ \ds_{0}^{V_{\alpha}} = \iota_{V_{\alpha}}^{U_{\alpha}},
    \end{equation}
    and $\ul{\theta'_{\alpha}} = \theta_{0}|_{V_{\alpha} \times I_{\alpha}}$, where $\iota_{V_{\alpha}}^{U_{\alpha}}: \M|_{V_{\alpha}} \rightarrow \M|_{U_{\alpha}}$ is the canonical inclusion. Let 
    \begin{equation}
    \theta_{\alpha} = \iota_{\alpha} \circ \theta'_{\alpha}: (\M \times \R)|_{V_{\alpha} \times I_{\alpha}} \rightarrow \M,
    \end{equation}
    where $\iota_{\alpha}: \M|_{U_{\alpha}} \rightarrow \M$ is the inclusion. Now, let us define the open set
    \begin{equation}
        D := \bigcup_{\alpha \in S} V_{\alpha} \times I_{\alpha}.
    \end{equation}
    We claim that $D$ is a flow domain on $M$. We have to check that for each $m \in M$, the set $D^{(m)}$ defined by (\ref{eq_Dmsubsetdomain}) is an open interval containing $0$. Let $S_{m} := \{ \alpha \in S \mid m \in V_{\alpha} \}$. Then 
    \begin{equation}
        D^{(m)} = \bigcup_{\alpha \in S_{m}} I_{\alpha}. 
    \end{equation}
    Since $S_{m} \neq \emptyset$, this is an open interval containing $0$. Hence $D$ is a flow domain on $M$, and $D \subseteq D_{0}$. Since $\{ V_{\alpha} \times I_{\alpha} \}_{\alpha \in S}$ is an open cover of $D$, we can now define $\theta: (\M \times \R)|_{D} \rightarrow \M$ to be the unique map satisfying $\theta|_{V_{\alpha} \times I_{\alpha}} = \theta_{\alpha}$ for each $\alpha \in S$. One only has to check that for each $\alpha,\beta \in S$, the maps $\theta_{\alpha}$ and $\theta_{\beta}$ agree on the intersection 
    \begin{equation}
        (V_{\alpha} \times I_{\alpha}) \cap (V_{\beta} \times I_{\beta}) = (V_{\alpha} \cap V_{\beta}) \times (I_{\alpha} \cap I_{\beta}).
    \end{equation}
    It is easy to see that both restrictions define a graded smooth map satisfying (\ref{eq_pivotalODEUVlocal}) for $U = U_{\alpha} \cap U_{\beta}$, $V = V_{\alpha} \cap V_{\beta}$, $I = I_{\alpha} \cap I_{\beta}$, and $\phi = \iota_{V}^{U}$, hence they must be equal by Lemma \ref{lem_thetalocal}. 

    Finally, to check that $\theta$ satisfies (\ref{eq_thetaquationalmostfinal}), simply apply both sides on a general function $f \in \C^{\infty}_{\M}(M)$ and compare restrictions of both sides to $V_{\alpha} \times I_{\alpha}$ for all $\alpha \in S$, using the definition of $\theta_{\alpha}$. The fact that $\ul{\theta} = \theta_{0}|_{D}$ follows immediately from the proof of Proposition \ref{tvrz_underlyingflow}. 
    \end{proof}

    \subsection{Uniqueness of the solution} \label{subsec_uniqueness}
    \begin{tvrz} \label{tvrz_uniqueness}
    Let $\theta: \D \rightarrow \M$ and $\theta': \D' \rightarrow \M$ both satisfy the equations (\ref{eq_thetaquationalmostfinal}). Then 
    \begin{equation}
        \theta|_{D \cap D'} = \theta'|_{D \cap D'}
    \end{equation}
    \end{tvrz}
    \begin{proof}
        Since $D \cap D'$ is always a flow domain, we may without the loss of generality assume that $D = D'$. The underlying map $\ul{\theta}$ is necessarily the restriction of the ordinary flow of $X_{0}$, so $D \subseteq D_{0}$ and $\ul{\theta} = \theta_{0}|_{D}$. In particular, $\ul{\theta} = \ul{\theta}'$. 
        
        For any $(m,t) \in D$, we will find $V \in \Op_{m}(M)$ together with an open interval $I \in \Op_{t}(\R)$, such that $V \times I \subseteq D$, and 
        \begin{equation} \label{eq_thetatheta'VtimesIcomperison}
            \theta|_{V \times I} = \theta'|_{V \times I},
        \end{equation}
        which will show the claim. Let $(m,t) \in D$ be any given point. Since $D^{(m)} \in \Op_{0}(\R)$ is an open interval, we have $(m,s) \in D$ for all $s \in [0,t]$. For each such $(m,s)$, there is $V_{s} \in \Op_{m}(M)$, $U_{s} \in \Op_{\theta_{0}(m,s)}(M)$ and an open interval $I_{s} \in \Op_{s}(\R)$ such that $\theta_{0}(V_{s} \times I_{s}) \subseteq U_{s}$. Moreover, we can assume that we have local charts $(U_{s}, \varphi_{s})$ and $(V_{s}, \psi_{s})$ for $\M$. Since $\{m\} \times [0,t]$ is compact, we can use its open cover $\{ V_{s} \times I_{s} \}_{s \in [0,t]}$ to find a finite subdivision $0 = t_{0} < \cdots < t_{N} = t$ together with the following data:
        \begin{enumerate}[(1)]
            \item Graded local charts $\{ (U_{k},\varphi_{k}) \}_{k=0}^{N}$ for $\M$, where $U_{k} \in \Op_{\theta_{0}(m,t_{k})}(M)$.
            \item A graded local chart $(V,\psi)$ for $\M$, where $V \in \Op_{m}(M)$.
            \item A collection $\{ I_{k} \}_{k=0}^{N}$ of open intervals covering $[0,t]$, such that $I_{k} \in \Op_{t_{k}}(\R)$ and $I_{k+1} \cap I_{k} \neq \emptyset$ for all $k \in \{0,\dots,N-1\}$. 
            \item For each $k \in \{0,\dots,N\}$, one has $\theta_{0}(V \times I_{k}) \subseteq U_{k}$.
        \end{enumerate}
        See the following figure to get the intuition:
        \begin{center}
            \includegraphics{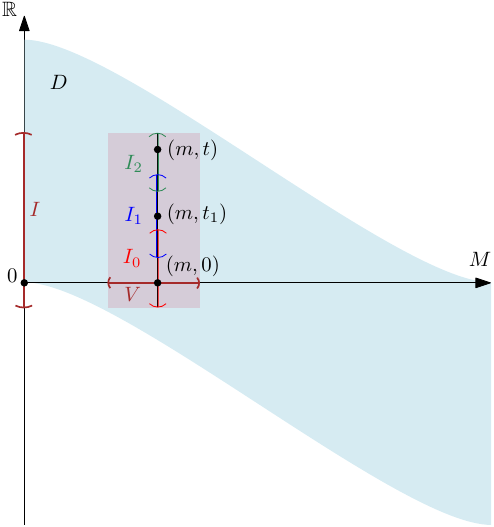}
        \end{center}
        The flow domain $D$ is blue, the rectangle $V \times I$ we are looking for is pink. For each $k \in \{0,\dots,N\}$, let us define $I'_{k} := \bigcup_{j=0}^{k} I_{j}$. Let us prove by induction on $k$ that
        \begin{equation}
            \theta|_{V \times I'_{k}} = \theta'|_{V \times I'_{k}}.
        \end{equation}
        For $k = 0$, the result follows from the uniqueness claim of Lemma \ref{lem_thetalocal}, where we can use the graded local chart $(U_{0},\varphi_{0})$, $V \in \Op_{m}(U)$ and $I_{0} \in \Op_{0}(\R)$, and the assumption that both $\theta$ and $\theta'$ satisfy (\ref{eq_thetaquationalmostfinal}). Next, let $k > 0$. By induction hypothesis, we have $\theta|_{V \times I'_{k-1}} = \theta'|_{V \times I'_{k-1}}$. One must only argue that $\theta|_{V \times I_{k}} = \theta'|_{V \times I_{k}}$. Pick any $t'_{0} \in I_{k-1} \cap I_{k}$. It follows from the induction hypothesis that
        \begin{equation}
            \theta|_{V \times I_{k}} \circ \ds_{t'_{0}}^{V} = \theta|_{V \times I'_{k-1}} \circ \ds^{V}_{t'_{0}} = \theta'|_{V \times I'_{k-1}} \circ \ds^{V}_{t'_{0}}  = \theta'|_{V \times I_{k}} \circ \ds^{V}_{t'_{0}},
        \end{equation}
        where $\ds^{V}_{t'_{0}}$ is the section valued at $t'_{0}$. Since both $\ul{\theta}$ and $\ul{\theta}'$ have to be the restrictions of $\theta_{0}$, and $\theta|_{V \times I_{k}}$ and $\theta'|_{V \times I_{k}}$ both satisfy (\ref{eq_pivotalODEUVlocal}) for the same ``initial condition map'' $\phi: \M|_{V} \rightarrow \M|_{U_k}$, they must be equal by Lemma \ref{lem_thetalocal}. This finishes the induction step.

        If we now consider the open interval $I := I'_{N} \in \Op_{t}(\R)$, we have $V \times I \subseteq D$ such that (\ref{eq_thetatheta'VtimesIcomperison}) holds. As $(m,t) \in D$ was arbitrary, this shows that $\theta = \theta'$. 
    \end{proof}
    \subsection{Finding the maximal flow domain} \label{subsec_maximaldomain}
    By Proposition \ref{tvrz_therexistsflow}, there exists a flow domain $\D = (\M \times \R)|_{D}$ and a graded smooth map $\theta: \D \rightarrow \M$ satisfying the equations (\ref{eq_thetaquationalmostfinal}), with $D \subseteq D_{0}$. Let $D_{\max}$ be the union of all such (underlying) flow domains $D$. 

    \begin{tvrz} \label{tvrz_thereismaximal}
        $D_{\max}$ is a flow domain contained in $D_{0}$. 
        
        Moreover, there exists a unique solution $\theta: (\M \times \R)|_{D_{\max}} \rightarrow \M$ of the equations (\ref{eq_thetaquationalmostfinal}) maximal in the following sense: Suppose $D'$ is any other flow domain and $\theta': (\M \times \R)|_{D'} \rightarrow \M$ satisfying the equations (\ref{eq_thetaquationalmostfinal}). Then $D' \subseteq D_{\max}$ and $\theta' = \theta|_{D'}$. 
    \end{tvrz}
    \begin{proof}
        We can write $D_{\max} = \cup_{\alpha \in J} D_{\alpha}$, where for each $\alpha \in J$, we have $\theta_{\alpha}: (\M \times \R)|_{D_{\alpha}} \rightarrow \M$ satisfying (\ref{eq_thetaquationalmostfinal}). Since $D_{\alpha} \subseteq D_{0}$ for all $\alpha \in J$, we have $D_{\max} \subseteq D_{0}$. Moreover, for each $m \in M$, one has $D_{\max}^{(m)} = \cup_{\alpha \in J} D_{\alpha}^{(m)}$. Since any union of intervals containing $0$ is an interval containing $0$, this proves that $D_{\max}$ is a flow domain. 

        Next, $\{ D_{\alpha} \}_{\alpha \in J}$ forms an open cover of $D_{\max}$ and the maps $\theta_{\alpha}$ agree on the overlaps by Proposition \ref{tvrz_uniqueness}. There is thus a unique graded smooth map $\theta: (M \times \R)|_{D_{\max}} \rightarrow \M$ satisfying $\theta|_{D_{\alpha}} = \theta_{\alpha}$ for each $\alpha \in J$. It follows from the construction that $\theta$ also satisfies the equations (\ref{eq_thetaquationalmostfinal}). It also follows immediately that $\theta$ is maximal in the sense described in the statement. 
    \end{proof}

    So far, we know that $D_{\max}$ is a subset of $D_{0}$, the flow domain of the ordinary flow $\theta_{0}$ of the underlying vector field $X_{0}$. It turns out that it is necessarily the same.

    \begin{tvrz} \label{tvrz_DmaxisD0}
        One has $D_{\max} = D_{0}$. 
    \end{tvrz}
    \begin{proof}
        Suppose that $(m_{0},s_{0}) \in D_{0}$ is not in $D_{\max}$. Let $\theta: (\M \times \R)|_{D_{\max}} \rightarrow \M$ be the maximal solution constructed above. 
        
        First, we will construct a strip $S := V \times I$, where $V \in \Op_{m_{0}}(M)$ and $I$ is an open interval containing $[0,s_{0}]$, together with a map $\theta_{S}: \M|_{V} \times I \rightarrow \M$ satisfying
        \begin{equation} \label{eq_thetaequationintheproof}
            (1 \otimes \partial_{t}) \circ \theta^{\ast}_{S} = \theta^{\ast}_{S} \circ X, \; \; \theta_{S} \circ \ds_{0}^{V} = \iota_{V},
        \end{equation}
        where $\iota_{V}: \M|_{V} \rightarrow \M$ is the inclusion and $\ds_{0}^{V}: \M|_{U} \rightarrow \M|_{V} \times I$ is the zero section. Again, one can find a subdivision $0 = t_{0} < \dots < t_{N} = s_{0}$ and data (1) - (4) as in the proof of Proposition \ref{tvrz_uniqueness}, except $U_{k} \in \Op_{\theta_{0}(m_{0},t_{k})}(M)$, $V \in \Op_{m_{0}}(M)$ and $\{ I_{k} \}_{k=0}^{N}$ covers $[0,s_{0}]$. Let $I'_{k} = \cup_{j=0}^{k} I_{j}$. We can choose the intervals so that $I'_{k} \cap I_{k+1} = I_{k} \cap I_{k+1}$, that is $I_{j} \cap I_{k+1} = \emptyset$ for $j < k$. 

        Let us now use the induction on $k \in \{0,\dots,N\}$ to construct a map $\theta^{(k)}: \M|_{V} \times I'_{k} \rightarrow \M$ satisfying (\ref{eq_thetaequationintheproof}). Then let $S := V \times I'_{N}$ and $\theta_{S} := \theta^{(N)}$. We will skip the details. For $k = 0$, use Lemma \ref{lem_thetalocal} directly to construct $\theta^{(0)}: \M|_{V} \times I_{0} \rightarrow \M$ satisfying (\ref{eq_thetaequationintheproof}). For $k > 0$, assume the existence of $\theta^{(k-1)}$ and use it to determine the ``initial condition'' at $t'_{0} \in I_{k-1} \cap I_{k}$. Use Lemma \ref{lem_thetalocal} and this initial condition to construct $\vartheta: \M|_{V} \times I_{k} \rightarrow \M$. The uniqueness assertion of Lemma \ref{lem_thetalocal} then ensures that $\vartheta$ and $\theta^{(k-1)}$ agree on the overlap $V \times (I_{k-1} \cap I_{k})$, hence they glue to the unique map $\theta^{(k)}: \M|_{V} \times I'_{k} \rightarrow \M$ satisfying (\ref{eq_thetaequationintheproof}). 

        We have constructed $\theta_{S}: \M|_{V} \times I \rightarrow \M$ satisfying (\ref{eq_thetaequationintheproof}), where $S = V \times I$ for $V \in \Op_{m_{0}}(M)$ and open interval $I$ containing $[0,s_{0}]$. In particular, observe that 
        \begin{equation}
        D'_{\max} := D_{\max} \cup S
        \end{equation}
        is a flow domain on $M$ containing $(m_{0},s_{0})$. Moreover, observe that 
        \begin{equation}
            \theta|_{D_{\max} \cap S} = \theta_{S}|_{D_{\max} \cap S}.
        \end{equation}
        This follows from the fact that $D_{\max} \cap S$ is a flow domain on $V \subseteq M$ and both maps satisfy (\ref{eq_thetaquationalmostfinal}). By Proposition \ref{tvrz_uniqueness}, they must coincide. There is thus a unique $\theta': (\M \times \R)|_{D'_{\max}} \rightarrow \M$, such that 
        \begin{equation}
            \theta'|_{D_{\max}} = \theta, \; \; \theta'|_{S} = \theta_{S}.
        \end{equation}
        It follows form the construction that $\theta'$ satisfies (\ref{eq_thetaquationalmostfinal}). Since $D_{\max} \subsetneq D'_{\max}$, this contradicts the maximality. Hence $D_{\max} = D_{0}$ and the proof is finished.
    \end{proof}
    \subsection{Proving the flow equations \& Conclusion} \label{subsec_flowequations}
    By Proposition \ref{tvrz_thereismaximal}, there is a unique maximal solution $\theta: (\M \times \R)|_{D} \rightarrow \M$ of the equations (\ref{eq_thetaquationalmostfinal}). We will henceforth write just $D$ for the flow domain of $\theta_{0}$. It remains to argue that it fits into the second commutative diagram in (\ref{eq_flowdiagrams}). For each $t \in \R$, we have an open (and possibly empty) set (\ref{eq_D(t)subset}) and we can consider a graded smooth map $\theta_{(t)}: \M|_{D_{(t)}} \rightarrow \M|_{D_{(-t)}}$ defined by (\ref{eq_thetatmap}).
    
    Now, fix $s \in \R$ so that $D_{(s)} \neq \emptyset$, and consider the open subset 
    \begin{equation} \label{eq_Dsopenset}
        D_{s} := \{ (m,t) \in D_{(s)} \times \R \mid (\theta_{0}(m,s),t) \in D \} \subseteq D_{(s)} \times \R.
    \end{equation}
    Equivalently, we have $(m,t) \in D_{s}$, iff $((m,s),t) \in D'$, see Lemma \ref{lem_ordinaryflow}. Observe that $D_{s}$ is a flow domain on $D_{(s)}$. We can now construct a pair of maps:
    \begin{align}
        \theta_{[s]} := & \ \theta \circ (\1_{\M} \times +) \circ \da \circ (\ds_{s} \times \1_{\R})|_{D_{s}}, \label{eq_defthetasmap} \\ 
        \theta'_{[s]} := & \ \theta \circ (\theta \times \1_{\R}) \circ (\ds_{s} \times \1_{\R})|_{D_{s}}. \label{eq_defthetasprimemap}
    \end{align}
    We will argue that those maps satisfy the same set of differential equations with the same initial conditions, hence they must be equal. The fact that $s$ is arbitrary will then imply the commutativity of the second diagram in (\ref{eq_flowdiagrams}).
    \begin{lemma} \label{lem_thetasequations}
        For each $s \in \R$ with $D_{(s)} \neq \emptyset$, the map $\theta_{[s]}$ satisfies the equations
        \begin{equation} \label{eq_thetasequations}
            (1 \otimes \partial_{t}) \circ \theta_{[s]}^{\ast} = \theta_{[s]}^{\ast} \circ X, \; \; \theta_{[s]} \circ \ds_{0}^{(s)} = \iota_{(-s)} \circ \theta_{(s)},
        \end{equation}
        where $\ds_{0}^{(s)}: \M|_{D_{(s)}} \rightarrow (\M \times \R)|_{D_{s}}$ is the zero section, $\theta_{(s)}: \M|_{D_{(s)}} \rightarrow \M|_{D_{(-s)}}$ is the map (\ref{eq_thetatmap}), and $\iota_{(-s)}: \M|_{D_{(-s)}} \rightarrow \M$ is the canonical inclusion.
    \end{lemma}
    \begin{proof}
        Let $L_{s}^{+}: \R \rightarrow \R$ be the left translation by $s \in \R$, that is $L_{s}^{+}(t) = s + t$. One has
        \begin{equation} \label{eq_maptodefinethetasdifferent}
            (\1_{\M} \times +) \circ \da \circ (\ds_{s} \times \1_{\R})|_{D_{s}} = (\1_{\M} \times L_{s}^{+})|_{D_{s}},
        \end{equation}
        and both sides land in the open submanifold $(\M \times \R)|_{D}$. In particular, the graded smooth map $\theta_{[s]}$ is well-defined and one can thus write 
        \begin{equation} \label{eq_thetasasleftranslation}
            \theta_{[s]} = \theta \circ (\1_{\M} \times L_{s}^{+})|_{D_{s}}.
        \end{equation}
        We leave the verification of (\ref{eq_maptodefinethetasdifferent}) to the reader. Recall that $\partial_{t}$ is a left-invariant vector field on $(\R,+)$, that is $L_{s}^{+}$-related to itself. Consequently, one has 
        \begin{equation} \label{eq_1otimespartialLstrelation}
            1 \otimes \partial_{t} \sim_{\1_{\M} \times L_{s}^{+}} 1 \otimes \partial_{t}.
        \end{equation}
        Observe that $(\1_{\M} \times L_{s}^{+})(D_{s}) \subseteq D$, see \textit{(p2)} of Theorem \ref{thm_fundamentalordinary}. Then $1 \otimes \partial_{t} \in \X_{\M \times \R}(D_{s})$ is $(\1_{\M} \times L_{s}^{+})|_{D_{s}}$-related to $1 \otimes \partial_{t} \in \X_{\M \times \R}(D)$. But since $1 \otimes \partial_{t} \in \X_{\M}(D)$ is $\theta$-related to $X \in \X_{\M}(M)$, this implies immediately that $1 \otimes \partial_{t} \in \X_{\M \times \R}(D_{s})$ is $\theta_{[s]}$-related to $X$. But this is precisely the first equation in (\ref{eq_thetasequations}). To prove the second equation in (\ref{eq_thetasequations}), utilize (\ref{eq_thetasasleftranslation}) together with the simple observation
        \begin{equation} \label{eq_leftranslationsections}
        (\1_{\M} \times L_{s}^{+})|_{D_{s}} \circ \ds_{0}^{(s)} = \ds_{s}.
        \end{equation}
        The rest is just the definition (\ref{eq_thetatmap}) of $\theta_{(s)}$. 
    \end{proof}
    \begin{lemma} \label{lem_thetasprimeequations}
        The statement of Lemma \ref{lem_thetasequations} remains true when $\theta_{[s]}$ is replaced by $\theta'_{[s]}$. 
    \end{lemma}
    \begin{proof}
        First, observe that $(\ds_{s} \times \1_{\R})|_{D_{s}}: (\M \times \R)|_{D_{s}} \rightarrow ((\M \times \R) \times \R)|_{D'}$, so the map $\theta'_{[s]}$ is well-defined. We encourage the reader to make the following simple observations:
        \begin{enumerate}[(1)]
            \item $1 \otimes \partial_{t} \in \X_{\M \times \R}(D_{s})$ is $(\ds_{s} \times \1_{\R})|_{D_{s}}$-related to $1_{\M \times \R} \otimes \partial_{t} \in \X_{(\M \times \R) \times \R}(D')$.
            \item $1_{\M \times \R} \otimes \partial_{t}$ is $(\theta \times \1_{\R})$-related to $1 \otimes \partial_{t} \in \X_{\M \times \R}(D)$. 
            \item $1 \otimes \partial_{t} \in \X_{\M \times \R}(D)$ is $\theta$-related to $X \in \X_{\M}(M)$.
        \end{enumerate}
        Since $\theta'_{[s]}$ is the composition of these three maps, we find that $1 \otimes \partial_{t} \in \X_{\M \times \R}(D_{s})$ is $\theta'_{[s]}$-related to $X$. This is the first equation in (\ref{eq_thetasequations}). Next, observe that
        \begin{equation}
            (\ds_{s} \times \1_{\R})|_{D_{s}} \circ \ds_{0}^{(s)} = (\ds_{s}, 0^{(s)}_{\M}),
        \end{equation}
        where $0^{(s)}_{\M}: \M|_{D_{(s)}} \rightarrow \R$ is the composition of the zero $0: \{ \ast \} \rightarrow \R$ with the unique terminal arrow into $\{ \ast \}$. Consequently, one has
        \begin{equation}
            (\theta \times \1_{\R}) \circ (\ds_{s} \times \1_{\R})|_{D_{s}} \circ \ds_{0}^{(s)} = (\theta \times \1_{\R}) \circ (\ds_{s}, 0^{(s)}_{\M}) = (\iota_{(-s)} \circ \theta_{(s)}, 0_{\M}^{(s)}) = \ds_{0} \circ \iota_{(-s)} \circ \theta_{(s)}.
        \end{equation}
        By composing this equation with $\theta$ and using the second part of (\ref{eq_thetaquationalmostfinal}), we get $\theta'_{[s]} \circ \ds_{0}^{(s)} = \iota_{(-s)} \circ \theta_{(s)}$, that is the second equation of (\ref{eq_thetasequations}). This finishes the proof.
    \end{proof}
    
    \begin{lemma}  \label{lem_thetathetaprimesameequationsoffsetgeneral}
           Let $D \subseteq M \times \R$ be any flow domain on $M$. Let $\cN$ be a graded manifold equipped with a degree zero vector field $X \in \X_{\cN}(N)$. Let $\phi: \M \rightarrow \cN$ be an arbitrary graded smooth map. Suppose $\theta,\theta': (\M \times \R)|_{D} \rightarrow \cN$ are graded smooth maps with the same underlying smooth map $\theta_{0}: D \rightarrow N$, such that $1 \otimes \partial_{t} \in \X_{\M \times \R}(D)$ and $X$ are both $\theta$-related and $\theta'$-related, and 
           \begin{equation}
               \theta \circ \ds_{0} = \theta' \circ \ds_{0} = \phi,
           \end{equation}
           where $\ds_{0}: \M \rightarrow (\M \times \R)|_{D}$ is the zero section.  Then $\theta = \theta'$.       
    \end{lemma}
    \begin{proof}
        The proof is a straightforward modification of the one of Proposition \ref{tvrz_uniqueness}.
    \end{proof}

    \begin{lemma} \label{lem_twomus}
        Let $\theta_{0}: D \rightarrow M$ be a flow on $M$. Suppose that $\mu,\mu': ((\M \times \R) \times \R)|_{D'} \rightarrow \M$ is a pair of graded smooth maps over the same underlying smooth map $\mu_{0}: D' \rightarrow M$, where $D' \subseteq (M \times \R) \times \R$ is the open subset (\ref{eq_Dprimesubset}). 

        Then $\mu = \mu'$, iff $\mu \circ (\ds_{s} \times \1_{\R})|_{D_{s}} = \mu' \circ (\ds_{s} \times \1_{\R})|_{D_{s}}$ for every $s \in \R$ with $D_{(s)} \neq \emptyset$. The open sets $D_{(s)}$ and $D_{s}$ are defined by (\ref{eq_D(t)subset}) and (\ref{eq_Dsopenset}), respectively. 
    \end{lemma}
    \begin{proof}
        For each $((m_{0},s_{0}),t_{0}) \in D'$, find $V \in \Op_{m_{0}}(M)$, $U \in \Op(M)$ and open intervals $I \in \Op_{s_{0}}(\R)$, $J \in \Op_{t_{0}}(\R)$, such that $(V \times I) \times J \subseteq D'$ and $\mu_{0}((V \times I) \times J) \subseteq U$. We can assume that we have local charts $(U,\varphi)$ and $(V,\psi)$ for $\M$. Note that for each $s \in I$, one has $V \times J \subseteq D_{s}$. We also have the induced local charts $(\psi \times \1_{I}) \times \1_{J}$ for $((\M \times \R) \times \R)|_{D'}$  and $\psi \times \1_{J}$ for $(\M \times \R)|_{D_{s}}$. The assumption now implies that
        \begin{equation}
            \hat{\mu} \circ (\hat{\ds}_{s} \times \1_{\R})|_{V \times J} = \hat{\mu}' \circ (\hat{\ds}_{s} \times \1_{\R})|_{V \times J},
        \end{equation}
        for each $s \in I$, where $\hat{\mu},\hat{\mu}': (\hat{V}^{(n_{j})} \times I) \times J \rightarrow \hat{U}^{(n_{j})}$ and $\hat{\ds}_{s}: \hat{V}^{(n_{j})} \rightarrow \hat{V}^{(n_{j})} \times J$ are the corresponding local representatives. It is easy to see that this implies $\hat{\mu} = \hat{\mu}'$. Since $((m_{0},s_{0}),t_{0}) \in D'$ was arbitrary, this proves that $\mu = \mu'$. 
    \end{proof}
    We can now combine all of the auxiliary statements to prove the main statement.
    \begin{tvrz} \label{tvrz_flowequationsfinal}
        Let $\theta: (\M \times \R)|_{D} \rightarrow \M$ be the maximal solution of (\ref{eq_thetaquationalmostfinal}). Then 
        \begin{equation}
            \theta \circ (\1_{\M} \times +) \circ \da|_{D'} = \theta \circ (\theta \times \1_{\R})|_{D'}.
        \end{equation}
        In other words, the second diagram in (\ref{eq_flowdiagrams}) commutes. Note that for a maximal flow domain, one has $D^{\bullet} = D'$ and thus $\D^{\bullet} = ((\M \times \R) \times \R)|_{D'}$, see Lemma \ref{lem_ordinaryflow} for the notation. 
    \end{tvrz}
    \begin{proof}
        For each $s \in \R$ with $D_{(s)} \neq \emptyset$, we have considered a pair of maps $\theta_{[s]}$ and $\theta'_{[s]}$ defined by (\ref{eq_defthetasmap}, \ref{eq_defthetasprimemap}). In Lemma \ref{lem_thetasequations} and Lemma \ref{lem_thetasprimeequations}, we have shown that they satisfy the equations (\ref{eq_thetasequations}). We would like to utilize Lemma \ref{lem_thetathetaprimesameequationsoffsetgeneral} for the flow domain $D_{s}$ on $D_{(s)}$ to conclude that $\theta_{[s]} = \theta'_{[s]}$. We have to argue that their underling smooth maps coincide. For any $(m,t) \in D_{s}$, one has 
        \begin{equation}
            \ul{\theta_{[s]}}(m,t) = \theta_{0}(m, s + t) = \theta_{0}(\theta_{0}(m,s), t) = \ul{\theta'_{[s]}}(m,t),
        \end{equation}
        where we have used the fact that $\ul{\theta} = \theta_{0}$ is the ordinary flow generated by $X_{0}$. Since $\theta_{[s]} = \mu \circ (\ds \times \1_{\R})|_{D_{s}}$ and $\theta'_{[s]} = \mu' \circ (\ds \times \1_{\R})|_{D_{s}}$ for
        \begin{equation}
            \mu := \theta \circ (\1_{\M} \times +) \circ \da|_{D'}, \; \; \mu' := \theta \circ (\theta \times \1_{\R})|_{D'},
        \end{equation}
        we can utilize Lemma \ref{lem_twomus} as soon as we verify that $\ul{\mu} = \ul{\mu'}$. But this again follows immediately from the fact that $\ul{\theta} = \theta_{0}$ is a flow of $X_{0}$. This finishes the proof  
    \end{proof}
    \textbf{Conclusion:} Let us now collect all of the pieces to form the proof of Theorem \ref{thm_funamental_extraordinary} for $k = 0$. 
    
    Given a vector field $X \in \X_{\M}(M)$ of degree zero, first construct a maximal flow $\theta_{0}: D_{0} \rightarrow M$ of the underlying vector field $X_{0} \in \X_{M}(M)$. Note that the requirement $[X,X] = 0$ is superfluous. By Proposition \ref{tvrz_thereismaximal}, there exists a unique maximal solution $\theta: (\M \times \R)|_{D} \rightarrow \M$ of the equations (\ref{eq_thetaquationalmostfinal}). This means that $1 \otimes \partial_{t} \in \X_{\M \times \R}(D)$ is $\theta$-related to $X$ and $\theta$ fits into the first diagram in (\ref{eq_flowdiagrams}). It follows from Proposition \ref{tvrz_DmaxisD0} that $D = D_{0}$ and from the proof of Proposition \ref{tvrz_underlyingflow} it is clear that $\ul{\theta} = \theta_{0}$. Proposition \ref{tvrz_flowequationsfinal} then shows that $\theta$ fits into the second diagram in (\ref{eq_flowdiagrams}). In other words, $\theta: \D \rightarrow \M$ is a flow on $\M$ of degree $0$. By construction, it satisfies both properties $(i)$ and $(ii)$ in Theorem \ref{thm_funamental_extraordinary}. This finishes the proof.
    
    \section{Proof of theorem: nonzero degree} \label{app_proofneq0}

    The proof of Theorem \ref{thm_funamental_extraordinary} becomes much simpler if $X$ has \textbf{nonzero degree} $k$. For the one thing, we need only proof part \textit{(i)}, see the commentary under the statement of the theorem. That is, we must show the existence of a unique degree $k$ flow $\theta : \D \to \M$ such that $1 \otimes \partial_\tau \sim_{\theta} X$, where $\tau$ is the (only) coordinate on $\R[-k]$. Recall that the flow domain of degree $k$ is $\D = \M \times \R[-k]$. The idea is that the flow can be expressed, in the weak sense, as the exponential
    \begin{equation}\label{eq_flow_exp}
        \theta^\ast = \exp(\tau X).
    \end{equation}
    In the rest of the section we will explain and formalize this rather vague statement. Our discussion will differ somewhat based on whether $k$ is odd or even. We begin with the simpler case.

    \subsection{Odd degree} \label{subsec_odd}
    Let $X$ be a vector field on $\M$ of odd degree $k := |X|$, such that $[X,X] = 2X^2 = 0$. Let $h \in \C^\infty_\D(U)$ be an arbitrary graded function for some $U \in \Op(M)$, where we slightly abuse notation by writing $\C^\infty_\D(U)$ instead of $\C^\infty_\D(U\times \{\ast\})$. Since $|\tau| = -k$ is also odd, one can decompose $h$ uniquely as
    \begin{equation}
        h = p_1^\ast(h^\prime) + \tau \, p_1^\ast(h^{\prime\prime}),
    \end{equation}
    for some $h^\prime, h^{\prime \prime} \in \C^\infty_\M(U)$, where $p_1 : \D \to \M$ is the canonical projection. Note that we obtain
    \begin{equation}\label{eq_odd_flow_decomp}
        h^\prime = \ds_0^\ast(h), \quad h^{\prime \prime} = \left(\ds_0^\ast \circ (1 \otimes \partial_\tau)\right) (h),
    \end{equation}
    where $\ds_0 := (\1_\M, 0_\M) : \M \to \D$, see Definition \ref{def_graded_flow}. Thus, if $\theta : \D \to \M$ is a flow of $X$, for any $f \in \C^\infty_\M(U)$ there must be
    \begin{equation}
    \begin{split}
        \theta^\ast(f) &= p_1^\ast(\ds_0^\ast(\theta^\ast(f))) + \tau \, p_1^\ast\left(\left(\ds_0^\ast \circ (1 \otimes \partial_\tau) \circ \theta^\ast\right)(f)\right) \\
        &= p_1^\ast(f) + \tau \, p_1^\ast(X(f)), \label{eq_odd_flow}
    \end{split}
    \end{equation}
    where we used the first flow diagram (\ref{eq_flowdiagrams}) and the relation $1 \otimes \partial_\tau \sim_\theta X$. More precisely, we used that $(1\otimes\partial_\tau)|_{U} \sim_{\theta|_{U}}X|_{U}$, but for clarity we did not write the restrictions.

    Let us verify that (\ref{eq_odd_flow}) defines a graded smooth map. Firstly, for every $U \in \Op(M)$ we must have a graded algebra morphism $\theta^\ast : \C^\infty_\M(U) \to \C^\infty_\D(U)$. Linearity is clear, and direct computation verifies that for any $f,g \in \C^\infty_\M(U)$ there is
    \begin{equation}
        \theta^\ast(fg) = \theta^\ast(f) \theta^\ast(g) + O(\tau^2),
    \end{equation}
    where the $O(\tau^2)$ term is in fact zero, since $|\tau| = -k$ is odd. Secondly, these component maps must form a sheaf morphism $\theta^\ast : \C^\infty_\M \to \ul{\theta}_\ast \C^\infty_\D$, but naturality with respect to restrictions is clear. Lastly, the induced stalk map $[f]_{\ul{\theta}(x)} \mapsto [\theta^\ast(f)]_{x}$ must be a local ring morphism for any $x \in D $, see \cite[Definitions 2.24 and 3.20]{Vysoky:2022gm}. This amounts to showing that if $\ul{f}(\ul{\theta}(x)) = 0$ then $\ul{\theta^\ast(f)}(x) = 0$, but this is so:
    \begin{equation}
        \ul{\theta^\ast(f)}(x) =  \ul{p_1^\ast(f)} + \ul{\tau \, p_1^\ast(X(f))} = \ul{f}(\ul{p_1}(x)) + 0 = \ul{f}(\ul{\theta}(x)) = 0,
    \end{equation}
    Where we use that $\ul{\theta} = \ul{p_1}$.

    We are left to argue that (\ref{eq_odd_flow}) defines a flow of $X$. The first flow diagram in (\ref{eq_flowdiagrams}) follows immediately from the definition. Let us redraw the second flow diagram here:
    \begin{equation}\label{eq_second_flow_diag_odd}
        \begin{tikzcd}
            \left(\M \times \R[-k](\rho)\right) \times \R[-k](\sigma)
            \arrow{d}{\da} 
            \arrow{r}{\theta \times \1_{\R}} 
                & \M \times \R[-k](\tau) 
                \arrow{dd}{\theta}\\
            \M \times (\R[-k](\rho) \times \R[-k](\sigma)) 
            \arrow{d}{\1_{\M} \times +}
                &{} \\
            \M \times \R[-k](\tau) \arrow{r}{\theta}
                & \M.
        \end{tikzcd},
    \end{equation}
    where we now explicitly indicate the graded coordinate on $\R[-k]$. This is so that we can pullback a function $f \in \C^\infty_\M(U)$ through both paths of the diagram. Let us also not explicitly write pullbacks by the canonical projection, i.e. we will write $f$ instead of $p_1^\ast(f)$. One path through the diagram (\ref{eq_second_flow_diag_odd}) reads:
    \begin{equation}\label{eq_odd_flow_diag_path1}
        (\theta \times \1_{\R})^\ast(\theta^\ast(f)) = (\theta \times \1_{\R})^\ast \left( f + \tau \, X(f)  \right) = f + \rho X(f) + \sigma \, \left(X(f) + 
        \rho \, X^2(f) \right),
    \end{equation}
    while the second path reads:
    \begin{equation}
        (\1_\M \times +)^\ast(\theta^\ast(f)) = (\1_\M \times +)^\ast \left( f + \tau \, X(f)\right) = f + (\rho + \sigma)\, X(f).
    \end{equation}
    We see that (\ref{eq_second_flow_diag_odd}) commutes because $X$ is by assumption a homological vector field. Finally, we must verify that $X$ is the infinitesimal generator of $\theta$. For a function $f \in \C^\infty_\M(M)$ we find
    \begin{equation}
        \left((1 \otimes \partial_\tau) \circ \theta^\ast \right)(f) = (1 \otimes \partial_\tau) \left(f + \tau \, X(f) \right) = X(f),
    \end{equation}
    and
    \begin{equation}\label{eq_odd_flow_infinitesimal_generator}
        \theta^\ast(X(f)) = X(f) + \tau \, X^2(f) = X(f),
    \end{equation}
    where we use the homological property of $X$ once more. This finishes the proof of Theorem \ref{thm_funamental_extraordinary} in the case $|X|$ odd. Some remarks are in order:

    \begin{rem}
        In (\ref{eq_odd_flow}) we found an explicit prescription for the flow $\theta$ of an odd vector field $X$. It agrees with the claim (\ref{eq_flow_exp}) that the flow ``is an exponential'' $\theta^\ast = \exp(\tau \, X)$ in the weak sense if we write the exponential as its Taylor expansion and recall that only terms with $\tau$ to the power of $0$ or $1$ survive.
    \end{rem} 

    \begin{rem}
        Suppose $X$ is not homological, i.e. $X^2 \neq 0$. We could still define the graded smooth map $\theta : \D \to \M$ by (\ref{eq_odd_flow}), and it would satisfy the first flow diagram in (\ref{eq_flowdiagrams}). However, it would not satisfy the second flow diagram, see (\ref{eq_odd_flow_diag_path1}). Furthermore, $1 \otimes \partial_\tau$ would not be $\theta$-related to $X$, see (\ref{eq_odd_flow_infinitesimal_generator}). Indeed, we knew this could not be so: see Remark \ref{rem_homological}.
    \end{rem}

    \subsection{Even nonzero degree} \label{subsec_even}

    Let $X$ be a vector field on $\M$ of even nonzero degree $|X| = k$. The proof follows the same lines as in the odd case, only details are somewhat more involved since $\tau$ does not square to zero. We would like to write a general function $h \in \C^\infty_{\M}(U)$ as
    \begin{equation}\label{eq_even_nozero_general_function}
        h = \sum_{\ell = 0}^\infty \frac{1}{\ell!} \, p_1^\ast(h_{\ell}) \, \tau^\ell,
    \end{equation}
    for some unique functions $h_{\ell} \in \C^\infty_\M(U)$. Here the factor $\frac{1}{\ell!}$ is added simply for convenience. However, the right hand side of (\ref{eq_even_nozero_general_function}) is not actually defined, so let us give it some meaning. Suppose $V \in \Op(U)$ is some coordinate open set with coordinates $x^i, \xi^\mu$ on $\M|_{V}$, hence we have coordinates $x^i, \xi^\mu$ and $\tau$ on $\D|_V$. For the coordinate representative of $h|_{V}$, we can write
    \begin{equation}
        \hat{h} = \sum_{(\fp, \ell)} \frac{1}{\ell!}\hat{h}_{\fp, \ell} \, \xi^{\fp} \,\tau^\ell,
    \end{equation}
    where $(\fp, \ell) \in \ol{\N}_{|h|}^{n_\ast + 1}$, see section \textbf{\ref{subsection_setting_the_notation}} for the notation. In particular, it is fully and uniquely determined by the collection of (ordinary) smooth functions $\hat{h}_{\fp, \ell}$. By (\ref{eq_even_nozero_general_function}) we actually mean that
    \begin{equation}
        (\hat{h}_\ell)_{\fp} := \hat{h}_{\fp, \ell}, 
    \end{equation}
    for any $\ell \in \{0,1,\dotsc\}$ and any $\fp$ such that $(\fp, \ell) \in \ol{\N}_{|h|}^{n_\ast + 1}$, and also for any coordinate subset $V \in \Op(U)$. Similarly as in (\ref{eq_odd_flow_decomp}), we can obtain the graded functions $h_{\ell}$ from $h$ via
    \begin{equation}\label{eq_even_component_functions}
      h_\ell = (\ds_0^\ast \circ (1 \otimes \partial_\tau)^\ell)(h),
    \end{equation}
    which can easily be seen in local coordinates. Suppose $\theta$ is the flow of $X$. Then for any $f \in \C^\infty_\M(U)$ we would have
    \begin{equation}\label{eq_even_theta_components}
        [\theta^\ast(f)]_\ell = (\ds_0^\ast \circ (1 \otimes \partial_\tau)^\ell \circ \theta^\ast)(f) = (\ds_0^\ast \circ \theta^\ast \circ X^\ell )(f) = X^\ell(f),
    \end{equation}
    hence
    \begin{equation}\label{eq_even_nonzero_flow}
        \theta^\ast(f) = \sum_{\ell = 0}^\infty \frac{1}{\ell!} \, p_1^\ast\left(X^\ell (f)\right)  \tau^\ell,
    \end{equation}
    where the meaning of the notation is explained in the text above. Note that to improve readability, we write $X(f)$ instead of the correct $X|_{U}(f)$. We are left to argue that (\ref{eq_even_nonzero_flow}) indeed defines a flow of $X$. As in the odd case, the first step is to show that (\ref{eq_even_nonzero_flow}) in fact defines a graded smooth map $\theta : \D \to \M$, for $\ul{\theta}$ trivial. For this we need to make sure that the multiplication of infinite sums of the form (\ref{eq_even_nozero_general_function}) behaves as expected.
    \begin{lemma}\label{lemma_even_multiplying}
        Let $h,h^\prime \in \C^\infty_\D(U)$. Then
        \begin{equation}
            hh^\prime = \sum_{\ell = 0}^\infty \frac{1}{\ell!}p_1^\ast\left(\sum_{r = 0}^\ell \binom{\ell}{r}  h_{\ell - r} h^\prime_r \,\right) \tau^\ell.
        \end{equation}
    \end{lemma}
    \begin{proof}
        Recall (\ref{eq_even_nozero_general_function}) and (\ref{eq_even_component_functions}). Under that notation we have
        \begin{equation}
            hh^\prime = \sum_{\ell = 0}^\infty \frac{1}{\ell!} \, p_1^\ast\left((hh^\prime)_\ell\right) \tau^\ell,
        \end{equation}
        where
        \begin{equation}
        \begin{split}
            (hh^\prime)_\ell &= \left(\ds_0^\ast\circ (1 \otimes \partial_\tau)^\ell\right)(hh^\prime) \\
            &= \ds_0^\ast \left(\sum_{r = 0}^\ell \binom{\ell}{r} (1 \otimes \partial_\tau)^{\ell - r}(h) \cdot (1 \otimes \partial_\tau)^r(h^\prime) \right) \\
            &= \sum_{r = 0}^\ell \binom{\ell}{r}  h_{\ell - r} h^\prime_r,
        \end{split}
        \end{equation}
        which shows the claim. In derivation, we used the fact that the vector field $1 \otimes \partial_\tau$ is even.
        \end{proof}
    As an immediate consequence we obtain
    \begin{equation}
        \theta^\ast(fg) = \sum_{\ell = 0}^\infty \frac{1}{\ell!} \, p_1^\ast\left(X^\ell (fg)\right)  \tau^\ell = \sum_{\ell = 0}^\infty \frac{1}{\ell!} \, p_1^\ast\left(\sum_{r = 0}^\ell \binom{\ell}{r} X^{\ell - r}(f)X^{r}(g)\right)  \tau^\ell = \theta^\ast(f) \, \theta^\ast(g),
    \end{equation}
    for any $f,g \in \C^\infty_\M(U)$, where the last equality follows from Lemma \ref{lemma_even_multiplying} and (\ref{eq_even_theta_components}). The reader is left to verify that $\theta^\ast$ is linear and behaves naturally with respect to restrictions, both of which can be shown similarly as in Lemma \ref{lemma_even_multiplying}. We are left to argue that the induced stalk map $[f]_{\ul{\theta}(x)} \mapsto [\theta^\ast(f)]_x$ is a local ring morphism for any $x \in D$, which is the same as showing that $\ul{f}(\ul{\theta}(x)) = 0$ implies $\ul{\theta^\ast(f)}(x) = 0$. But from the definition it is clear that $\ul{\theta^\ast(f)}(x) = 0$ if and only if $\ul{p_1^\ast(f)}(x) = 0$, and there is $\ul{p_1^\ast(f)}(x) = \ul{f}(\ul{p_1}(x)) = \ul{f}(\ul{\theta}(x))$. We conclude that (\ref{eq_even_nonzero_flow}) defines a graded smooth map $\theta : \D \to \M$. 
    
    Let us see if $\theta$ is the flow of $X$. The first flow diagram in (\ref{eq_flowdiagrams}) follows immediately from the definition and (\ref{eq_even_component_functions}) for $\ell = 0$. We also have $1 \otimes \partial_\tau \sim_\theta X$ per the following lemma:
    \begin{lemma}
        Let $h \in \C^\infty_\D(U)$, then for any $r \in \N$ there is
        \begin{equation}
            (1 \otimes \partial_\tau)^r(h) = \sum_{\ell = 0}^\infty \frac{1}{\ell!} \, p_1^\ast(h_{\ell + r}) \, \tau^\ell.
        \end{equation}
        As a consequence, $1 \otimes \partial_\tau \sim_\theta X$.
    \end{lemma}
    \begin{proof}
        Recall (\ref{eq_even_nozero_general_function}) and (\ref{eq_even_component_functions}). Then we can write
        \begin{equation}
            \left( (1\otimes \partial_\tau)^r (h) \right)_\ell = \left(\ds_0^\ast \circ (1 \otimes \partial_\tau)^{\ell + r}\right)(h) = h_{\ell + r},
        \end{equation}
        which shows the first claim. Now consider some $f \in \C^\infty_\M(M)$ and write
        \begin{equation}
            \left((1 \otimes \partial_\tau) \circ \theta^\ast\right)(f) = \sum_{\ell = 0}^\infty \frac{1}{\ell!} \, p_1^\ast(X^{\ell+1}(f)) \, \tau^\ell = \theta^\ast(X(f)),
        \end{equation}
        where we used the definition (\ref{eq_even_nonzero_flow}).
    \end{proof}
    The most involved is the verification of the second flow diagram in (\ref{eq_flowdiagrams}), which is redrawn for convenience as (\ref{eq_second_flow_diag_odd}) in the previous section. Due to triviality of $\ul{\theta}$ in the case $|X| \neq 0$ we will dispense with the explicit writing of the associator diffeomorphism for the rest of the section. Thus consider the graded manifold
    \begin{equation}
        \D^\bullet := \M \times \R[-k](\rho) \times \R[-k](\sigma),
    \end{equation}
    where we explicitly denote the graded coordinate on each copy of the shifted vector space $\R[-k]$. Let $h \in \C^\infty_{\D^\bullet}(U)$ be a general function. Similarly as in (\ref{eq_even_nozero_general_function}) we can write
    \begin{equation}\label{eq_even_general_function_on_Dbul}
        h = \sum_{\ell,r = 0}^{\infty} \frac{1}{\ell! r!} \, p_1^\ast(h_{\ell,r}) \, \rho^\ell \sigma^r,
    \end{equation}
    for unique graded functions $h_{\ell, r} \in \C^\infty_{\M}(U)$. One obtains these graded functions by
    \begin{equation}
        h_{\ell,r} = \left(\ds_{00}^\ast \circ (1 \otimes \partial_\rho \otimes 1)^\ell \circ (1 \otimes 1 \otimes \partial_\sigma)^r \right)(h),
    \end{equation}
    where $\ds_{00} : \M \to \D^{\bullet}$ is the ``double zero section'', $\ds_{00} = (\1_\M, 0_\M, 0_\M)$  where $0_\M : \M \to \R[-k]$ is the map given by $0_{\M}^\ast(\tau) = 0$.

    \begin{lemma}
        Let $f \in \C^\infty_\M(U)$. Then
        \begin{equation}
            \left((\theta \times \1_{\R[-k]})^\ast \circ \theta^\ast \right)(f) = \sum_{\ell,r = 0}^\infty \frac{1}{\ell!\, r!} \, p_1^\ast(X^{\ell + r}(h_{\ell,r})) \, \rho^\ell \sigma^r = \left((\1_{\R[-k]} \times +)^\ast \circ \theta^\ast \right)(f)
        \end{equation}
    \end{lemma}
    \begin{proof}
        In the notation of (\ref{eq_even_general_function_on_Dbul}) we have
         \begin{equation}
    \begin{split}
       [\left((\theta \times \1_{\R[-k]})^\ast \circ \theta^\ast \right)(f)]_{\ell,r} &= \left(\ds_{00}^\ast \circ (1 \otimes \partial_\rho \otimes 1)^\ell \circ (1 \otimes 1 \otimes \partial_\sigma)^r \circ (\theta \times \1_{\R[-k]})^\ast \circ \theta^\ast \right)(f) \\
            &= \left(\ds_{00}^\ast \circ (1 \otimes \partial_\rho \otimes 1)^\ell \circ (\theta \times \1_{\R[-k]})^\ast \circ (1 \otimes \partial_\tau)^r  \circ \theta^\ast \right)(f) \\
            &= \left(\ds_{00}^\ast \circ (\theta \times \1_{\R[-k]})^\ast \circ (X \otimes 1)^\ell \circ \theta^\ast\circ  X^r \right)(f) \\
            &= \left(\ds_{00}^\ast \circ (\theta \times \1_{\R[-k]})^\ast \circ \theta^\ast\circ  X^{\ell+r} \right)(f) \\
            &= X^{\ell+r}(f),
   \end{split}
   \end{equation}
        where we made use of the fact that $X \otimes 1 \sim_{\theta} X$ by Proposition \ref{tvrz_Xisinvariantunderitsflow}. We also invite the reader to verify that $ \theta \circ (\theta \times \1_{\R[-k]}) \circ \ds_{00} = \1_{\M}$.  Similarly we have
        \begin{equation}
        \begin{split}
            [\left((\1_{\M} \times +)^\ast \circ \theta^\ast \right)(f)]_{\ell,r} &= \left(\ds_{00}^\ast \circ (1 \otimes \partial_\rho \otimes 1)^\ell \circ (1 \otimes 1 \otimes \partial_\sigma)^r \circ (\1_{\M} \times +)^\ast \circ \theta^\ast  \right)(f)\\
            &= \left(\ds_{00}^\ast \circ (\1_{\M} \times +)^\ast \circ \theta^\ast \circ X^{\ell + r} \right)(f) \\
            &=X^{\ell+r}(f),
        \end{split}
        \end{equation}
        where now we use that $\partial_\tau$ is both a left- and right-invariant vector field on $\R[-k]$. In other words, both $1 \otimes \partial_\sigma \sim_+ \partial_\tau$ and $\partial_\rho \otimes 1 \sim_+ \partial_\tau$. This was used together with $X \otimes 1 \sim_\theta X$ and $\theta \circ (\1_{\M}\times +) \circ \ds_{00} = \1_{\M}$. This ends the proof.
    \end{proof}
    \begin{rem}
        We have shown that (\ref{eq_even_nonzero_flow}) gives an explicit prescription for the flow of a vector field $X$ of nonzero even degree. Similarly as in the odd degree case, it agrees with the claim (\ref{eq_flow_exp}) that $\theta^\ast = \exp(\tau X)$ in the weak sense, when one views the exponential as its Taylor series.
    \end{rem}

    \section{Invariance under flows: proofs} \label{app_invariance}

    This is a section containing technical statements and their proofs utilized in Section \ref{sec_invariance}. 

    \begin{lemma} \label{lem_relatedwithiV}
        Let $\M$ and $\cN$ be graded manifolds. Let $U \in \Op(M)$ and $V \in \Op(N)$ be open subsets together with a graded smooth map $\phi: \M|_{U} \rightarrow \cN|_{V}$. Let $X \in \X_{\M}(U)$ and $Y \in \X_{\cN}(V)$ be two vector fields. Then $X \sim_{\phi} Y$, if and only if 
        \begin{equation} \label{eq_relatedwithiV}
            X \circ \phi^{\ast} \circ \iota_{V}^{\ast} = \phi^{\ast} \circ Y \circ \iota_{V}^{\ast},
        \end{equation}
        where $\iota_{V}: \M|_{V} \rightarrow \M$ is the canonical inclusion, that is $\iota^{\ast}_{V}(f) = f|_{V}$ for all $f \in \C^{\infty}_{\cN}(N)$.
        \end{lemma}
        \begin{proof}
        The only if direction is obvious. For the if direction, use a partition of unity. 
         \end{proof}
         \begin{lemma} \label{lem_invariancecomposedwiths't}
             Let $X,Y \in \X_{\M}(M)$ with $|X| = 0$, and let $\theta: (\M \times \R)|_{D} \rightarrow \M$ be the flow of $X$. Let $Y \otimes 1 \in \X_{\M \times \R}(D)$ be the canonical lift of $Y$. Then $Y \otimes 1 \sim_{\theta} Y$, iff 
             \begin{equation} \label{eq_Yinvariantpulledbys't}
                 \ds^{\ast}_{t} \circ (Y \otimes 1) \circ \theta^{\ast} = \ds^{\ast}_{t} \circ \theta^{\ast} \circ X
             \end{equation}
             for all $t \in \R$ with $D_{(t)} \neq \emptyset$, where $\ds_{t}: \M|_{D_{(t)}} \rightarrow (\M \times \R)|_{D}$ is the section valued at $t$. 
         \end{lemma}
         \begin{proof}
             The only if part is trivial. Hence suppose that (\ref{eq_Yinvariantpulledbys't}) holds for all $t \in \R$ with $D_{(t)} \neq \emptyset$. For each $(m_{0},t_{0}) \in D$, find $V \in \Op_{m_{0}}(M)$, $U \in \Op(M)$ and an open interval $I \in \Op_{t_{0}}(\R)$ such that $\theta_{0}(V \times I) \subseteq U$. We can assume that we have a local charts $(V,\psi)$ and $(U,\phi)$ for $\M$. Note that $\ul{\ds_{t_{0}}}(V) \subseteq V \times I \subseteq D$, hence $V \subseteq D_{(t_{0})}$. The equation (\ref{eq_Yinvariantpulledbys't}) then implies
             \begin{equation}
                 \hat{\ds}^{\ast}_{t} \circ (\hat{Y}' \otimes 1) \circ \hat{\theta}^{\ast} = \hat{\ds}^{\ast}_{t} \circ \hat{\theta}^{\ast} \circ \hat{Y},
             \end{equation}
             for all $t \in I$, where $\hat{\theta}: \hat{V}^{(n_{j})} \times I \rightarrow \hat{U}^{(n_{j})}$ is the local representative of $\theta$ and $\hat{Y} \in \X_{(n_{j})}(\hat{U})$ and $\hat{Y}'\in \X_{(n_{j})}(\hat{V})$ are the local representatives of the vector field Y. Finally $\hat{\ds}_{t}: \hat{V}^{(n_{j})} \rightarrow \hat{V}^{(n_{j})} \times I$ is the local representative of $\ds_{t}$. Since the corresponding pullback just evaluates the component functions at $t \in I$, this clearly implies $(\hat{Y}' \otimes 1) \circ \hat{\theta}^{\ast} = \hat{\theta}^{\ast} \circ \hat{Y}$. But this is just a local expression of the required equation. Since $(m_{0},t_{0}) \in D$ was arbitrary, this proves the claim.
         \end{proof}
         \subsection{Proof of theorem: degree zero} \label{subsec_invariance0}
         This is the proof of the complicated bit of Theorem \ref{thm_invariance}. In this subsection, we consider $X \in \X_{\M}(M)$ of degree zero. Suppose that $[X,Y] = 0$. We need to prove (\ref{eq_invariantVF}). 

        First, let us show that the claim is equivalent to a seemingly more complicated statement.
        \begin{lemma} \label{lem_invariantVFjinak}
            The equation (\ref{eq_invariantVF}) is equivalent to the equation
            \begin{equation} \label{eq_invariantVFjinak}
                \Delta_{-}^{\ast} \circ (\theta \times \1_{\R})^{\ast} \circ (Y \otimes 1) \circ \theta^{\ast} = p_{1}^{\ast} \circ Y,
            \end{equation}
            where $p_{1}: \D \rightarrow \M$ is the projection, and $\Delta_{-}: \D \rightarrow \D'$ is the restriction of the map (\ref{eq_Delta-}).
        \end{lemma}
        \begin{proof}
            Observe that $\ul{\Delta}_{-}(m,t) = ((m,t),-t)$, so if $(m,t) \in D$, then $(\theta_{0}(m,t),-t) \in D$, so $\Delta_{-}$ can indeed be viewed as a map from $\D$ to $\D'$. 

            Let us consider a graded smooth map $\phi := (\theta \times \1_{\R}) \circ \Delta_{-}: \D \rightarrow \D$. We claim that it satisfies the following identities:
            \begin{equation}
                p_{1} \circ \phi = \theta, \; \; p_{2} \circ \phi = - \circ p_{2}, \; \; \phi \circ \phi = \1_{\D},
            \end{equation}
            where $p_{2}: \D \rightarrow \R$ is the second projection. The first two properties follow immediately from the definition of $\phi$ and $\Delta_{-}$. For the last one, one finds 
            \begin{align}
                p_{1} \circ (\phi \circ \phi) = & \  \theta \circ \phi = \theta \circ (\theta \times \1_{\R}) \circ \Delta_{-} = \theta \circ (\1_{\M} \times +) \circ \da \circ \Delta_{-} = \theta \circ (\ds_{0} \circ p_{1}) = p_{1}, \\
                p_{2} \circ (\phi \circ \phi) = & \  (- \circ p_{2}) \circ \phi = (- \circ -) \circ p_{2} = \1_{\R} \circ p_{2} = p_{2},
            \end{align}
            where we have utilized (\ref{eq_flowdiagrams}) together with $(\1_{\M} \times +) \circ \da \circ \Delta_{-} = \ds_{0} \circ p_{1}$ and $- \circ - = \1_{\R}$. This proves that $\phi \circ \phi = \1_{\D}$. It is now easy to observe that (\ref{eq_invariantVFjinak}) is just (\ref{eq_invariantVF}) composed with $\phi^{\ast}$. Since this is an algebra isomorphism (it squares to the identity), the claim follows.  
        \end{proof}

        Let us now define an $\R$-linear map $\fM: \C^{\infty}_{\M}(M) \rightarrow \C^{\infty}_{\D}(D)$ by the formula
        \begin{equation} \label{eq_fMmap}
            \fM := \Delta_{-}^{\ast} \circ (\theta \times \1_{\R})^{\ast} \circ (Y \otimes 1) \circ \theta^{\ast} - p_{1}^{\ast} \circ Y.
        \end{equation}
        Thanks to the preceding lemma, we have to prove that $\fM = 0$, whenever $[X,Y] = 0$. We have to utilize the following technical statements:
         \begin{lemma} \label{lem_fMiszero}
             One has $\fM = 0$, if and only if $\fM$ satisfies the equations
             \begin{equation} \label{eq_Mvanishhessufficient}
                 (1 \otimes \partial_{t}) \circ \fM = 0, \; \; \ds_{0}^{\ast} \circ \fM = 0.
             \end{equation}
         \end{lemma}
         \begin{proof}
             One only has to prove the if part. Hence suppose that $\fM$ satisfies (\ref{eq_Mvanishhessufficient}). Let us utilize the fact that $D$ is a flow domain for the flow $\theta_{0} = \ul{\theta}$. Let $(m,t) \in D$ be arbitrary. Let us now use the same argument and notation as in the proof of Proposition \ref{tvrz_uniqueness} to find a subdivision $0 = t_{0} < \cdots < t_{N} = t$ together with the the data (1) - (4), except $U_{k} \in \Op_{\theta_{0}(m,t_{k})}(M)$. Let $f \in \C^{\infty}_{\M}(M)$ be arbitrary. We will now use the induction in $k$ to prove that 
             \begin{equation}
                 \fM(f)|_{V \times I'_{k}} = 0,
             \end{equation}
             for all $k \in \{0,\dots,N\}$. Since $V \times I'_{N}$ is an open neighborhood of $(m,t)$ and this point in $D$ was arbitrary, this will prove the claim. 
             
             Let $k = 0$. By utilizing the coordinate charts $(V,\psi)$ and $(U_{0},\varphi_{0})$, the equation $\fM(f)|_{V \times I_{0}} = 0$ is equivalent to 
             \begin{equation}
                 \hat{\fM}(\hat{f}) = 0,
             \end{equation}
             where $\hat{\fM}$ is the ``local representative'' of $\fM$, and $\hat{f}$ is the local representative of $f$. Now, $\hat{\fM}(\hat{f})$ is just a formal power series, each coefficient a smooth function on $\hat{V} \times I_{0}$. The assumption (\ref{eq_Mvanishhessufficient}) implies that this function is constant in the extra $\R$ variable and vanishes on $\hat{V} \times \{0\}$, hence it vanishes on entire $\hat{V} \times I_{0}$.

             Let $k > 0$. The induction hypothesis is that $\fM(f)|_{V \times I'_{k-1}} = 0$. It thus suffices to prove that $\fM(f)|_{V \times I_{k}} = 0$. One uses the same argument and the local charts $(V,\psi)$ and $(U_{k},\varphi_{k})$, except one has to choose a point $t'_{0} \in I_{k-1} \cap I_{k}$, and all the coefficient functions of $\hat{\fM}(\hat{f})$ vanish on $\hat{V} \times \{t'_{0}\}$ by the induction hypothesis. This finishes the induction step.
         \end{proof}

        It remains to argue that $\fM$ defined by (\ref{eq_fMmap}) indeed satisfies (\ref{eq_Mvanishhessufficient}). First, observe that
        \begin{equation}
            (1 \otimes \partial_{t}) \circ \Delta_{-}^{\ast} = \Delta^{\ast}_{-} \circ \big( (1 \otimes \partial_{t}) \otimes 1 - 1_{\M \times \R} \otimes \partial_{t} \big).
        \end{equation}
        This is easy to verify on coordinate functions. Using the fact that $1 \otimes \partial_{t} \sim_{p_{1}} 0$, we thus get 
        \begin{equation}
            (1 \otimes \partial_{t}) \circ \fM = \Delta_{-}^{\ast} \circ \big( (1 \otimes \partial_{t}) \otimes 1 - 1_{\M \times \R} \otimes \partial_{t} \big) \circ (\theta \times \1_{\R})^{\ast} \circ (Y \otimes 1) \circ \theta^{\ast}. 
        \end{equation}
        As a next step, note that $(1 \otimes \partial_{t}) \otimes 1$ is $(\theta \times \1_{\R})$-related to $X \otimes 1$ and $1_{\M \times \R} \otimes \partial_{t}$ is $(\theta \times \1_{\R})$-related to $1 \otimes \partial_{t}$. We can thus write 
        \begin{equation}
            (1 \otimes \partial_{t}) \circ \fM = \Delta_{-}^{\ast} \circ (\theta \times \1_{\R})^{\ast} \circ \big ((X \otimes 1) - (1 \otimes \partial_{t}) \big) \circ (Y \otimes 1) \circ \theta^{\ast}.
        \end{equation}
        Next, observe that and $Y \otimes 1$ commutes with $1 \otimes \partial_{t}$, so 
        \begin{equation}
            (1 \otimes \partial_{t}) \circ \fM = \Delta_{-}^{\ast} \circ (\theta \times \1_{\R})^{\ast} \circ \big ((X  \otimes 1)  \circ (Y \otimes 1) - (Y \otimes 1) \circ (1 \otimes \partial_{t}) \big) \circ \theta^{\ast}
        \end{equation}
        Since $1 \otimes \partial_{t} \sim_{\theta} X$, one can use $(1 \otimes \partial_{t}) \circ \theta^{\ast} = \theta^{\ast} \circ X$. At this moment, utilize Proposition \ref{tvrz_Xisinvariantunderitsflow}. Indeed, as $X$ is invariant under $\theta$, we have $\theta^{\ast} \circ X = (X \otimes 1) \circ \theta^{\ast}$. Consequently, one finds
        \begin{equation}
        \begin{split}
            (1 \otimes \partial_{t}) \circ \fM =  & \ \Delta_{-}^{\ast} \circ (\theta \times \1_{\R})^{\ast} \circ ( [X \otimes 1, Y \otimes 1] \circ \theta^{\ast} \\
            = & \ \Delta_{-}^{\ast} \circ (\theta \times \1_{\R})^{\ast} \circ ( [X,Y] \otimes 1) \circ \theta^{\ast} = 0,
        \end{split}
        \end{equation}
        where we have used the assumption $[X,Y] = 0$. This proves the first equation in (\ref{eq_Mvanishhessufficient}). For the second one, observe that 
        \begin{equation}
        (\theta \times \1_{\R}) \circ \Delta_{-} \circ \ds_{0} = (\theta \times \1_{\R}) \circ (\ds_{0} \times \1_{\R}) \circ \ds_{0} = \1_{\D} \circ \ds_{0} = \ds_{0}. 
        \end{equation}
        Moreover, one has $Y \sim_{\ds_{0}} Y \otimes 1$ and $p_{1} \circ \ds_{0} = \1_{\M}$. Hence
        \begin{equation}
        \begin{split}
            \ds_{0}^{\ast} \circ \fM = & \  \ds_{0}^{\ast} \circ \Delta_{-}^{\ast} \circ (\theta \times \1_{\R})^{\ast} \circ (Y \otimes 1) \circ \theta^{\ast} - \ds_{0}^{\ast} \circ p_{1}^{\ast} \circ Y \\
            = & \ \ds_{0}^{\ast} \circ (Y \otimes 1) \circ \theta^{\ast} - Y \\
            = & \ Y \circ (\ds_{0}^{\ast} \circ \theta^{\ast}) - Y = 0,
        \end{split}
        \end{equation}
        where we have used the first diagram in (\ref{eq_flowdiagrams}). This proves the second equation in (\ref{eq_fMmap}). Hence $\fM = 0$ by Lemma \ref{lem_fMiszero}, and thus $Y$ is invariant under the flow $\theta$ of $X$ by definition of $\fM$ and Lemma \ref{lem_invariantVFjinak}. This finishes the proof.
        \subsection{Proof of theorem: nonzero degree} \label{subsec_invariance1}
 
        When $|X| \neq 0$, the proof of Theorem \ref{thm_invariance} becomes rather simple. Following the notation we put down in Appendix \ref{app_proofneq0}, we shall again consider separately the cases when $|X|$ is odd and when it is even nonzero.
        
        Let $|X|$ be odd and let $Y$ be an arbitrary vector field such that $[X,Y] = 0$. Let $\theta$ be the flow of $X$. Then for any $f \in \C^\infty_\M(M)$ one has
        \begin{equation}
            (Y \otimes 1)(\theta^\ast(f)) = (Y \otimes 1)\left(p_1^\ast(f) + p_1^\ast(X(f))\,\tau \right) = p_1^\ast(Y(f)) + p_1^\ast(X(Y(f))) \, \tau = \theta^\ast(Y(f)),
        \end{equation}
        in other words $Y \otimes 1 \sim_{\theta} Y$.

        Now let $|X|$ be even. Recall in particular the notation given in (\ref{eq_even_nozero_general_function}). First, we need a small lemma.
    \begin{lemma}\label{lemma_even_x_invariant_under_x}
    For any $Y \in \X_\M(M)$ and $h \in \C^\infty_\D(M)$ there is 
    \begin{equation}\label{eq_even_x_times_one}
        (Y \otimes 1)(h) = \sum_{\ell = 0}^\infty \frac{1}{\ell!}\,p_1^\ast(Y(h_\ell))\, \tau^\ell.
    \end{equation}
\end{lemma}
\begin{proof}
   Denote the right hand side of (\ref{eq_even_x_times_one}) as $Z(h)$. Note that $[Z(h)]_\ell = Y(h_\ell)$, so $Z$ is a well-defined linear map of degree $|Z| = |Y|$. For any two $h,h^\prime \in \C^\infty_\D(M)$ one has
   \begin{equation}
   \begin{split}
       [Z(hh^\prime)]_\ell &= Y\left((hh^\prime)_\ell\right) = Y \left( \sum_{r = 0}^\ell \binom{\ell}{r} h_{\ell - r} h_r^\prime \right) \\
       &= \sum_{r = 0}^\ell\binom{\ell}{r} \left(Y(h_{\ell-r}) \, h^\prime_{r} + (-1)^{|Y||h_{\ell-r}|} h_{\ell - r} \, Y(h^\prime_\ell) \right) \\
       &= [Z(h)h^\prime + (-1)^{|Z||h|}h Z(h^\prime)]_\ell.
       \end{split}
   \end{equation}
    In the last step we used that $|Z| = |Y|$ and $|h_{\ell - r}| = |h| - (\ell - r)|\tau| = |h| \mod2$, since $|\tau|$ is odd. This shows $Z$ to be a vector field. 
    
    Let $h = p_1^\ast(f)$ for some $f \in \C^\infty_\M(M)$. Then $h_0 = f$ and $h_\ell = 0$ for all $\ell \geq 1$, and so $Z(p_1^\ast(f)) = p_1^\ast(Y(f))$. Similarly, let $h = p_2^\ast(g)$ for some graded function $g$ on $\R[-k]$. Then $h_\ell$ is a constant for every $\ell$, and so $Z(p_2^\ast(f)) = 0$. We have just shown that $Z \sim_{p_1} Y$ and $Z \sim_{p_2}0$, hence $Z = Y \otimes 1$. This proves (\ref{eq_even_x_times_one}).
\end{proof}
        Now let $Y$ be a vector field such that $[X,Y] = 0$. Then for any $f \in \C^\infty_\M(M)$ we can write
        \begin{equation}
            [\left((Y\otimes 1)\circ \theta^\ast\right)(f)]_\ell = Y(X^\ell(f)) = X^\ell(Y(f)) = [\theta^\ast(Y(f))]_\ell,
        \end{equation}
        where the first equality follows from Lemma \ref{lemma_even_x_invariant_under_x} and (\ref{eq_even_theta_components}). This shows $Y \otimes 1 \sim_\theta Y$.

        \section{Commuting flows: proofs} \label{app_commuting}
        \subsection{Facts about commuting domains} \label{subsec_factscommuting}
        This is a subsection dedicated to the proof of Lemma \ref{lem_commuting}. To simplify the notation, let us write 
        \begin{equation} \label{eq_D'12}
            D'_{12} := D'_{1} \cap \sigma^{-1}(D'_{2}). 
        \end{equation}

        Let us start by proving $(i)$. Let $m \in M$ be arbitrary. Clearly $((m,0),0) \in D'_{12}$. Since this set is open, there exists $U_{m} \in \Op_{m}(M)$ and open intervals $I_{m},J_{m} \in \Op_{0}(\R)$, such that $(U_{m} \times I_{m}) \times J_{m} \in D'_{12}$. Note that $M = \bigcup_{m \in M} U_{m}$. Let
        \begin{equation}
            \bbD := \bigcup_{m \in M} (U_{m} \times I_{m}) \times J_{m}.
        \end{equation}
        This is an open subset of $D'_{12}$. We must prove that it has the property (2) of Definition \ref{def_commutingdomain}. For each fixed $m \in M$, one has
        \begin{equation}
            \bbD^{(m)} = \bigcup_{\{ m' \mid m \in U_{m'} \}} I_{m'} \times J_{m'}.
        \end{equation}
        Since this is a union of rectangles containing $(0,0)$, a set certainly having the required property. Hence $\bbD$ is a commuting domain of $X$ and $Y$. 

        Let us continue with $(ii)$. Let $\bbD$ and $\bbD'$ be a pair of commuting domains of $X$ and $Y$. Clearly
        \begin{equation}
            (\bbD \cup \bbD')^{(m)} = \bbD^{(m)} \cup \bbD'^{(m)}, \; \; (\bbD \cap \bbD')^{(m)} = \bbD^{(m)} \cap \bbD'^{(m)},
        \end{equation}
        for all $m \in M$. But if $\bbD^{(m)}$ and $\bbD'^{(m)}$ have the property (2) of Definition \ref{def_commutingdomain}, then so does their union and intersection. The property (1) is obvious. The generalization of this statement to arbitrary finite intersections and arbitrary unions is immediate. 

        To prove $(iii)$, simply take the union of all commuting domains of $X$ on $Y$. By part $(ii)$, this is a commuting domain of $X$ and $Y$, and its maximality is obvious. The claim $(iv)$ is trivial. 

        To prove $(v)$, let $m \in M$ and suppose that there are open intervals $I,J \in \Op_{0}(\R)$ such that $((m,s),t) \in D'_{12}$ for all $(s,t) \in I \times J$. We must argue that $I \times J \subseteq \bbD^{(m)}$, where $\bbD$ is the maximal commuting domain of $X$ and $Y$. Pick $(s_{0},t_{0}) \in I \times J$. 
        
        The idea of the proof is the following: we shall find an open rectangle $I_{0} \times J_{0}$ containing both $(0,0)$ and $(s_{0},t_{0})$, and $U \in \Op_{m}(M)$, such that 
        \begin{equation}
        S := (U \times I_{0}) \times J_{0} \subseteq D'_{12}.
        \end{equation}
        It follows that $\bbD' := \bbD \cup S$ is a commuting domain. Since $\bbD$ is maximal, this forces $S \subseteq \bbD$ and thus $(s_{0},t_{0}) \in \bbD^{(m)}$, thus completing the proof.

        Let $R \subseteq I \times J$ denote the closed rectangle spanning from $(0,0)$ to $(s_{0},t_{0})$. For each $(s,t) \in R$, we have $((m,s),t) \in D'_{12}$ by assumption. Since this is an open subset, there exists $U' \in \Op_{m}(M)$ and open intervals $I' \in \Op_{s}(\R)$ and $J' \in \Op_{t}(\R)$, such that $(U' \times I') \times J' \subseteq D'_{12}$. In particular, we have $(s,t)$ in the open rectangle $I' \times J'$. In this way, we obtain a collection $\{ U_{\alpha} \}_{\alpha \in K}$ of open neighborhoods of $m$, and and open cover $\{ I_{\alpha} \times J_{\alpha} \}_{\alpha \in K}$ of $R$, such that $(U_{\alpha} \times I_{\alpha}) \times J_{\alpha} \subseteq D'_{12}$. Since $R$ is compact, there is a finite subset $K_{0} \subseteq K$, such that $\{ I_{\alpha} \times J_{\alpha} \}_{\alpha \in K_{0}}$ still covers $R$. Let $U := \bigcap_{\alpha \in K_{0}} U_{\alpha}$. Finally, there is certainly an open rectangle $I_{0} \times J_{0}$, such that 
        \begin{equation}
            R \subseteq I_{0} \times J_{0} \subseteq \bigcup_{\alpha \in K_{0}} I_{\alpha} \times J_{\alpha}.
        \end{equation}
        To see this, look at the following figure:
        \begin{center}
            \includegraphics[scale=1]{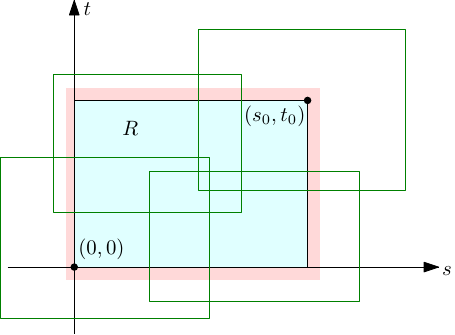}
        \end{center}
        The rectangle $R$ is in cyan, the finitely many covering rectangles are green, and the open rectangle $I_{0} \times J_{0}$ is pink. It follows from the construction that $S := (U \times I_{0}) \times J_{0} \subseteq D'_{12}$.

        Finally, The proof of $(vi)$ is somewhat similar. For any $((m,s),t) \in \bbD$, form a closed rectangle $R$ spanning from $(0,0)$ to $(s,t)$. Since $\bbD$ is a commuting domain, every point $(s',t') \in R$ satisfies $((m,s'),t') \in \bbD$. Since this is an open subset, we can use the same arguments as above to construct $U \in \Op_{m}(M)$ and open intervals $I,J \in \Op_{0}(\R)$ satisfying $(U \times I) \times J \subseteq \bbD$. 
        \subsection{Proof of theorem: twice degree zero} \label{subsec_proofcommuting1}
        Let us assume that $X,Y \in \X_{\M}(M)$ are both of degree zero. We want to prove Theorem \ref{thm_commuting}. 

        Let us first assume that $[X,Y] = 0$. Let $\bbD$ be the maximal commuting domain of the vector fields $X_{0}$ and $Y_{0}$. It exists and it is unique by Lemma \ref{lem_commuting}-$(iii)$. It follows immediately from Lemma \ref{lem_underlying} that $[X_{0},Y_{0}] = 0$. Hence by Theorem \ref{thm_commutingordinary}, their flows $\theta_{0}^{X}: D_{1} \rightarrow M$ and $\theta_{0}^{Y}: D_{2} \rightarrow M$ commute on $\bbD$. This means that on the level of underlying smooth maps, the diagram (\ref{eq_commutingdiagram}) commutes.

        Fix $((m_{0},s_{0}),t_{0}) \in \bbD$. By Lemma \ref{lem_commuting}-$(vi)$, we can find $U \in \Op_{m_{0}}(M)$ and open intervals $I,J \in \Op_{0}(\R)$, such that $((m_{0},s_{0}),t_{0}) \in (U \times I) \times J \subseteq \bbD$. We will prove that the restriction of the diagram (\ref{eq_commutingdiagram}) to $(U \times I) \times J$ commutes. Since $((m_{0},s_{0}),t_{0}) \in \bbD$ is arbitrary, this will be enough to conclude that $\theta^{X}$ and $\theta^{Y}$ commute on the maximal commuting domain $\sfD = ((\M \times \R) \times \R)|_{\bbD}$. 

        The proof is a compilation of several tricks already used in this paper, hence we will a bit sloppy with the notation. First, it suffices to prove that the compositions of both paths along the diagram with the map $(\ds_{s} \times \1_{\R})|_{U \times J}$, where $\ds_{s}: \M \rightarrow \M \times \R$ is a section valued at $s$,  commutes for each $s \in I$. This is an analogue of Lemma \ref{lem_twomus}. It is not difficult to argue that this boils down to proving the equation
        \begin{equation} \label{eq_commutinggdiagramateachs}
        \iota_{(-s)} \circ \theta^{X}_{(s)} \circ \theta^{Y}|_{U \times J} = \theta^{Y} \circ ((\iota_{(-s)} \circ \theta^{X}_{(s)}) \times \1_{\R})|_{U \times J},
        \end{equation}
        for all $s \in I$. $\theta_{(s)}: \M|_{D_{(s)}} \rightarrow \M|_{D_{(-s)}}$ is the map (\ref{eq_thetatmap}). We invite the reader to check that the composition is well-defined, e.g. note that $U \subseteq D_{(s)}$.
        
        Let $\mu_{(s)}$ and $\mu'_{(s)}$ denote the graded smooth maps from $(\M \times \R)|_{U \times J}$ to $\M$ corresponding to the left-hand side and to the right-hand side of (\ref{eq_commutinggdiagramateachs}), respectively. It turns out they satisfy the same differential equation.

        \begin{lemma}
            For each $s \in I$, the map $\mu_{(s)}$ satisfies the equations 
            \begin{equation}
                (1 \otimes \partial_{t}) \circ \mu_{(s)}^{\ast} = \mu_{(s)}^{\ast} \circ Y, \; \; \mu_{(s)} \circ \ds_{0} = \iota_{(-s)} \circ \theta_{(s)}^{X}|_{U},
            \end{equation}
            where $\ds_{0}: \M|_{U} \rightarrow (\M \times \R)|_{U \times J}$ is the zero section. 
            
            The same equations are true when $\mu_{(s)}$ is replaced by $\mu'_{(s)}$. 
        \end{lemma}
        \begin{proof}
            Since $[X,Y] = 0$, we know from Theorem \ref{thm_invariance} that that $Y$ is invariant under the flow $\theta^{X}$. In particular, we can utilize Proposition \ref{tvrz_invariancefordegreezero} and equation (\ref{eq_Yinvarethetatrelated}). Let us restrain from explicit writing of domains and restrictions of pullbacks.  One finds
            \begin{equation}
                (1 \otimes \partial_{t}) \circ \mu_{(s)}^{\ast} = (1 \otimes \partial_{t}) \circ (\theta^{Y})^{\ast} \circ (\theta_{(s)}^{X})^{\ast} = (\theta^{Y})^{\ast} \circ Y \circ (\theta^{X}_{(s)})^{\ast} = (\theta^{Y})^{\ast} \circ (\theta^{X}_{(s)})^{\ast} \circ Y  = \mu_{(s)}^{\ast} \circ Y.
            \end{equation}
            Utilizing the first diagram in (\ref{eq_flowdiagrams}), we get
            \begin{equation}
                \mu_{(s)} \circ \ds_{0} = \theta_{(s)}^{X} \circ \theta^{Y} \circ \ds_{0} = \theta_{(s)}^{X}.
            \end{equation}
            This proves the claims about $\mu_{(s)}$. For the other map, one has 
            \begin{equation}
            \begin{split}
                (1 \otimes \partial_{t}) \circ \mu'^{\ast}_{(s)} = & \ (1 \otimes \partial_{t}) \circ (\theta_{(s)}^{X} \times \1_{\R})^{\ast} \circ (\theta^{Y})^{\ast} = (\theta_{(s)}^{X} \times \1_{\R})^{\ast} \circ (1 \otimes \partial_{t}) \circ (\theta^{Y})^{\ast} \\
                = & \ (\theta_{(s)}^{X} \times \1_{\R})^{\ast} \circ (\theta^{Y})^{\ast} \circ Y = \mu'^{\ast}_{(s)} \circ Y.
            \end{split}
            \end{equation}
            Finally, since $(\theta_{(s)}^{X} \times \1_{\R}) \circ \ds_{0} = \ds_{0} \circ \theta_{(s)}^{X}$, we can use the first diagram in (\ref{eq_flowdiagrams}) to get 
            \begin{equation}
                \mu'_{(s)} \circ \ds_{0} = \theta^{Y} \circ (\theta^{X}_{(s)} \times \1_{\R}) \circ \ds_{0} = \theta^{Y} \circ \ds_{0} \circ \theta_{(s)}^{X} = \theta_{(s)}^{X}.
            \end{equation}
            This proves the claims about $\mu'_{(s)}$. 
        \end{proof}
        To prove the equation (\ref{eq_commutinggdiagramateachs}) it now suffices to invoke an analogue of Lemma \ref{lem_thetathetaprimesameequationsoffsetgeneral}. 
        
        Indeed, we have proved that the two maps $\mu_{(s)}$ and $\mu'_{(s)}$, defined on $(\M \times \R)|_{U \times J}$ and having the same underlying smooth map, both relate the vector fields $1 \otimes \partial_{t} \in \X_{\M \times \R}(U \times J)$ and $Y \in \X_{\M}(M)$. Moreover, $\mu_{(s)} \circ \ds_{0} = \mu'_{(s)} \circ \ds_{0} = \iota_{(-s)} \circ \theta_{(s)}^{X}|_{U}$. Since $U \times J$ is a flow domain, this implies
        \begin{equation}
            \mu_{(s)} = \mu'_{(s)},
        \end{equation}
        that is precisely (\ref{eq_commutinggdiagramateachs}). Since $s \in I$ was arbitrary, this finishes the proof of the fact that $\theta^{X}$ and $\theta^{Y}$ commute on the maximal commuting domain $\sfD$. 

        Note that the arguments used in the previous paragraph would fail if we would try to prove that $\theta^{X}$ and $\theta^{Y}$ commute on $((\M \times \R) \times \R)|_{D'_{12}}$, where $D'_{12}$ is the open set (\ref{eq_D'12}) of all points where the equation $\theta^{Y}_{0}( \theta^{X}_{0}(m,s),t) = \theta^{X}_{0}( \theta^{Y}_{0}(m,t),s)$ makes sense. This is because that for $D'_{12}$, the existence of $U$ and open intervals $I$ and $J$ as in Lemma \ref{lem_commuting}-$(vi)$ is not guaranteed. 

        Conversely, suppose that we have degree zero $X,Y \in \X_{\M}(M)$, such that $\theta^{X}$ and $\theta^{Y}$ commute on the maximal commuting domain $\sfD = ((\M \times \R) \times \R)|_{\bbD}$, where $\bbD$ is the maximal commuting domain of the underlying vector fields $X_{0}$ and $Y_{0}$. We must argue that $[X,Y] = 0$.  

        Fix $m \in M$. Since $((m,0),0) \in D'_{12}$, where $D'_{12}$ denotes the open subset (\ref{eq_D'12}), it follows from Lemma \ref{lem_commuting}-$(v)$ that $((m,0),0) \in \bbD$. By Lemma \ref{lem_commuting}-$(vi)$, there exists $U \in \Op_{m}(M)$ and open intervals $I,J \in \Op_{0}(\R)$, such that $(U \times I) \times J \subseteq \bbD$. The idea is the following. We assume that (\ref{eq_commutingdiagram}) commutes. Restrict the both maps forming the two paths through the diagram to the open subset $(U \times I) \times J$. We obtain the following equation of the pullback maps:
        \begin{equation}
            \sigma^{\ast} \circ (\theta^{Y} \times \1_{\R})^{\ast} \circ (\theta^{X})^{\ast} = (\theta^{X} \times \1_{\R})^{\ast} \circ (\theta^{Y})^{\ast},
        \end{equation}
        where we again do not explicitly write domains of pullbacks and restrictions to open subsets. Both sides are $\R$-linear maps mapping into $\C^{\infty}_{\sfD}((U \times I) \times J)$. Act on both sides by the vector field $1_{\M \times \R} \otimes \partial_{t} \in \X_{\sfD}((U \times I) \times J)$ and pull the result back by the zero section $\ol{\ds}_{0}: (\M \times \R)|_{U \times I} \rightarrow \sfD|_{(U \times I) \times J}$. 
        
        For the left-hand side, notice that $1_{\M \times \R} \otimes \partial_{t}$ is $\sigma$-related to $(1 \otimes \partial_{t}) \otimes 1 \in \X_{\sfD'}((U \times J) \times I)$, and since $1 \otimes \partial_{t}$ is $\theta^{Y}$-related to $Y$ by the definition of the flow, this vector field is $(\theta^{Y} \times \1_{\R})$-related to $Y \otimes 1 \in \X_{\D_{1}}(U \times I)$. We thus obtain 
        \begin{equation}
        \begin{split}
        (1_{\M \times \R} \otimes \partial_{t}) \circ \sigma^{\ast} \circ (\theta^{Y} \times \1_{\R})^{\ast} \circ (\theta^{X})^{\ast} = & \ \sigma^{\ast} \circ ((1 \otimes \partial_{t}) \otimes 1) \circ (\theta^{Y} \times \1_{\R})^{\ast} \circ (\theta^{X})^{\ast} \\
        = & \ \sigma^{\ast} \circ (\theta^{Y} \times \1_{\R})^{\ast} \circ (Y \otimes 1) \circ (\theta^{X})^{\ast}.
        \end{split}
        \end{equation}
        To pull this back by $\ol{\ds}_{0}^{\ast}$, observe that $\sigma \circ \ol{\ds}_{0} = (\ds_{0} \times \1_{\R})|_{U \times I}$. Utilizing (\ref{eq_flowdiagrams}), we find
        \begin{equation}
            \ol{\ds}_{0}^{\ast} \circ (1_{\M \times \R} \otimes \partial_{t}) \circ \sigma^{\ast} \circ (\theta^{Y} \times \1_{\R})^{\ast} \circ (\theta^{X})^{\ast} = (Y \otimes 1) \circ (\theta^{X})^{\ast}. 
        \end{equation}
        For the right-hand side, one uses the fact that $1_{\M \times \R} \otimes \partial_{t}$ and $1 \otimes \partial_{t}$ are $(\theta^{X} \times \1_{\R})$-related, hence 
        \begin{equation}
            (1_{\M \times \R} \otimes \partial_{t}) \circ (\theta^{X} \times \1_{\R})^{\ast} \circ (\theta^{Y})^{\ast} = (\theta^{X} \times \1_{\R})^{\ast} \circ (1 \otimes \partial_{t}) \circ (\theta^{Y})^{\ast} = (\theta^{X} \times \1_{\R})^{\ast} \circ (\theta^{Y})^{\ast} \circ Y,
        \end{equation}
        where the second equality uses the fact that $\theta^{Y}$ is the flow of $Y$. Pulling this back by $\ol{\ds}_{0}$ and using the first diagram in (\ref{eq_flowdiagrams}) for $\theta^{X}$, one obtains
        \begin{equation}
            \begin{split}
            \ol{\ds}_{0}^{\ast} \circ (1_{\M \times \R} \otimes \partial_{t}) \circ (\theta^{X} \times \1_{\R})^{\ast} \circ (\theta^{Y})^{\ast} = & \  \ol{\ds}_{0}^{\ast} \circ (\theta^{X} \times \1_{\R})^{\ast} \circ (\theta^{Y})^{\ast} \circ Y \\
            = & \ (\theta^{X})^{\ast} \circ \ds_{0}^{\ast} \circ (\theta^{Y})^{\ast} \circ Y = (\theta^{X})^{\ast} \circ Y. 
            \end{split}
        \end{equation}
        We have thus just proved the equation $(Y \otimes 1) \circ (\theta^{X})^{\ast} = (\theta^{X})^{\ast} \circ Y$. If we would now remember the restriction morphisms we have intentionally ignored, we would have just proved that the restriction of the vector field $Y \otimes 1$ to $U \times I$ is $\theta^{X}|_{U \times I}$-related to $Y$. But this immediately implies that $[X,Y] = 0$, using the same reasoning as in the easy part of the proof of Theorem \ref{thm_invariance}. 
        \subsection{Proof of theorem: at least one degree is non-zero} \label{subsec_proofcommuting2}
        Let $X,Y \in \X_{\M}(M)$ be two vector fields, where at least one of them has a non-zero degree. We want to prove Theorem \ref{thm_commuting} in this case.

        First, suppose that $Y$ is of a degree $\ell \neq 0$. To compare the results of pullbacks by both paths in (\ref{eq_commutingdiagram}), it suffices to compare them when composed with $\ol{\ds}_{0}^{\ast} \circ (1_{\M \times \R} \otimes \partial_{\tau'})^{r}$, where $\ol{\ds}_{0}: \D_{1} \rightarrow \sfD$ is the zero section, $\tau'$ is the coordinate function on $\R[-\ell]$ and $r \in \N_{0}$. The argument follows the one above the equation (\ref{eq_even_component_functions}). One has
        \begin{equation}
        \begin{split}
            \ol{\ds}_{0}^{\ast} \circ (1_{\M \times \R} \otimes & \partial_{\tau'})^{r} \circ \sigma^{\ast} \circ (\theta^{Y} \times \1_{\R})^{\ast} \circ (\theta^{X})^{\ast} = \\
            = & \ \ol{\ds}_{0}^{\ast} \circ \sigma^{\ast} \circ ((1 \otimes \partial_{\tau'})^{r} \otimes 1) \circ (\theta^{Y} \times \1_{\R})^{\ast} \circ (\theta^{X})^{\ast} \\
            = & \ (\ds_{0} \times \1_{\R})^{\ast} \circ (\theta^{Y} \times \1_{\R})^{\ast} \circ (Y \otimes 1)^{r} \circ (\theta^{X})^{\ast} \\
            = & \ (Y \otimes 1)^{r} \circ (\theta^{X})^{\ast}. 
        \end{split}
        \end{equation}
        On the other hand, one finds
        \begin{equation}
        \begin{split}
            \ol{\ds}_{0}^{\ast} \circ (1_{\M \times \R} \otimes \partial_{\tau'})^{r} \circ (\theta^{X} \times \1_{\R[-\ell]})^{\ast} \circ (\theta^{Y})^{\ast} = & \ (\theta^{X})^{\ast} \circ \ds_{0}^{\ast} \circ (1 \otimes \partial_{\tau'})^{r} \circ (\theta^{Y})^{\ast} \\
            = & \ (\theta^{X})^{\ast} \circ Y^{r}. 
        \end{split}
        \end{equation}
        We have thus argued that the commutativity of the diagram (\ref{eq_commutingdiagram}) is equivalent to the equation
        \begin{equation}
            (Y \otimes 1)^{r} \circ (\theta^{X})^{\ast} = (\theta^{X})^{\ast} \circ Y^{r}
        \end{equation}
        being true for all $r \in \N_{0}$. This is obviously equivalent to $Y$ being invariant under the flow $\theta^{X}$. Hence by Theorem \ref{thm_invariance}, this is equivalent to $[X,Y] = 0$. This finishes the proof.

        To prove the case where $X$ is of a non-zero degree $k \neq 0$, simply observe that the statement of Theorem \ref{thm_commuting} holds for $X$ and $Y$, iff it holds for $Y$ and $X$. This follows immediately from the fact that $\sigma^{-1} = \sigma$, so the diagram (\ref{eq_commutingdiagram}) for $X$ and $Y$ immediately gives the diagram for $Y$ and $X$, and vice versa. Since $[X,Y] = 0$ iff $[Y,X] = 0$, the observation follows.
        \end{document}